\documentclass{amsart}

\usepackage{mathabx} %has \lefttorightarrow (group action)
\usepackage{graphicx}
\usepackage{stmaryrd} %has \mapsfrom
\usepackage{tikz}
\usepackage{tikz-cd}
\usepackage{ bbold }
\usetikzlibrary{arrows}
\usepackage{verbatim}
\usepackage{amssymb,amsfonts}
\usepackage[all,arc]{xy}
\usepackage{enumerate}
\usepackage{mathrsfs}
\usepackage{xcolor}[2007/01/21]
\usepackage{hyperref}
\usepackage{soul}
\usepackage[all]{xy}
\usepackage[margin=1.0in]{geometry}
\usepackage[
   open,
   openlevel=2,
   atend
 ]{bookmark}[2011/12/02]
\usepackage{graphics}

\bookmarksetup{color=blue}

\theoremstyle{definition} 
\newtheorem{thm}{Theorem}[section]
\newtheorem{cor}[thm]{Corollary}
\newtheorem{prop}[thm]{Proposition}
\newtheorem{lem}[thm]{Lemma}
\newtheorem{conj}[thm]{Conjecture}

\theoremstyle{definition}
\newtheorem{defn}[thm]{Definition}

\newtheorem{rmk}[thm]{Remark}

\newtheorem{thmx}{Theorem}

\theoremstyle{remark}

\newcommand{\GL}{\mathrm{GL}}

\newcommand{\Aut}{\mathrm{Aut}}
\newcommand{\bZ}{\mathbb{Z}}
\newcommand{\bQ}{\mathbb{Q}}
\newcommand{\bR}{\mathbb{R}}
\newcommand{\bP}{\mathbb{P}}
\newcommand{\bA}{\mathbb{A}}
\newcommand{\bC}{\mathbb{C}}
\newcommand{\bF}{\mathbb{F}}

\newcommand{\mf}{\mathfrak}
\newcommand{\ms}{\mathscr}
\newcommand{\mc}{\mathcal}
\newcommand{\sm}{\smallsetminus}
\newcommand{\sub}{\subset}

\newcommand{\Spec}{\mathrm{Spec}}

\newcommand{\Mat}{\mathrm{Mat}}
\newcommand{\Prob}{\mathrm{Prob}}
\newcommand{\coker}{\mathrm{coker}}

\newcommand{\Inj}{\mathrm{Inj}}
\newcommand{\Stab}{\mathrm{Stab}}

\newcommand{\acts}{\lefttorightarrow}

\newcommand{\Mod}{\textbf{Mod}}
\newcommand{\ld}{\lambda}
\newcommand{\End}{\mathrm{End}}

\newcommand{\id}{\mathrm{id}}
\newcommand{\bs}{\boldsymbol}
\newcommand{\ol}{\overline}
\newcommand{\es}{\emptyset}

\newcommand{\ra}{\rightarrow}

\newcommand{\tra}{\twoheadrightarrow}

\newcommand{\hra}{\hookrightarrow}

\newcommand{\op}{\oplus}
\newcommand{\bop}{\bigoplus}

\newcommand{\Cl}{\mathrm{Cl}}

\newcommand{\Pic}{\mathrm{Pic}}

\newcommand{\IQ}{\textbf{IQ}}

\newcommand{\be}{\begin{enumerate}}
\newcommand{\ee}{\end{enumerate}}

\newcommand{\llb}{\llbracket}
\newcommand{\rrb}{\rrbracket}

\numberwithin{equation}{section}

\begin{document}

\title[Random matrices over complete discrete valuation rings]{Cohen-Lenstra distributions via random matrices over complete discrete valuation rings with finite residue fields}
\date{\today}
\author{Gilyoung Cheong and Yifeng Huang}
\address{Department of Mathematics, University of Michigan, 530 Church Street, Ann Arbor, MI 48109-1043, USA}
\email{gcheong@umich.edu, huangyf@umich.edu}

\begin{abstract}
Let $(R, \mf{m})$ be a complete discrete valuation ring with the finite residue field $R/\mf{m} = \bF_{q}$. Given a monic polynomial $P(t) \in R[t]$ whose reduction modulo $\mf{m}$ gives an irreducible polynomial $\ol{P}(t) \in \bF_{q}[t]$, we initiate the investigation of the distribution of $\coker(P(A))$, where $A \in \Mat_{n}(R)$ is randomly chosen with respect to the Haar probability measure on the additive group $\Mat_{n}(R)$ of $n \times n$ $R$-matrices. One of our main results generalizes two results of Friedman and Washington. Our other results are related to the distribution of the $\ol{P}$-part of a random matrix $\ol{A} \in \Mat_{n}(\bF_{q})$ with respect to the uniform distribution, and one of them generalizes a result of Fulman. We heuristically relate our results to a celebrated conjecture of Cohen and Lenstra, which predicts that given an odd prime $p$, not dividing $q$, any finite abelian $p$-group (i.e., $\bZ_{p}$-module) $H$ occurs as the $p$-part of the class group of a random imaginary quadratic field extension of $\bQ$ with a probability inversely proportional to $|\Aut_{\bZ}(H)|$. We review three different heuristics for the conjecture of Cohen and Lenstra, and they are all related to special cases of our main conjecture, which we prove as our main theorems. For proofs, we use some concrete combinatorial connections between $\Mat_{n}(R)$ and $\Mat_{n}(\bF_{q})$ to translate our problems about a Haar-random matrix in $\Mat_{n}(R)$ into problems about a random matrix in $\Mat_{n}(\bF_{q})$ with respect to the uniform distribution.
\end{abstract}

\maketitle

\subsection*{Standard notations} We write $p$ to mean an arbitrary prime (number) and $q$ an arbitrary prime power. We do not assume any relations between $p$ and $q$, unless specified otherwise. We fix arbitrary $n \in \bZ_{\geq 0}$, although we will often use it as index or let it go to infinity. By a \textbf{ring}, we mean a commutative ring with the multiplicative identity $1$. By a \textbf{distribution}, we mean a probability measure. Given an ideal $\mf{I}$ of a ring $R$ and a module $M$ over it, we define

\[M[\mf{I}^{\infty}] := \{x \in M : \mf{I}^{N}x = 0 \text{ for } N \gg 0\}\]

\

and call it the \textbf{$\mf{I}^{\infty}$-torsion} or the \textbf{$\mf{I}$-part} of $M$. If $M = M[\mf{I}^{\infty}]$, we call $M$ an \textbf{$\mf{I}^{\infty}$-torsion module}. For $t \in R$, we say \textbf{$t^{\infty}$-torsion} or \textbf{$t$-part} to mean $(t)^{\infty}$-torsion or $(t)$-part, and write $M[t^{\infty}] := M[(t)^{\infty}]$. We write $\Mat_{n}(R)$ to mean the set of $n \times n$ matrices over $R$, and $I_{n} \in \Mat_{n}(R)$ means the identity matrix. 

\

\section{Introduction}\label{intro}

\hspace{3mm} In number theory, an influential conjecture of Cohen and Lenstra \cite{CL} states that when $p$ is odd, a fixed finite abelian $p$-group $H$ occurs as the $p$-part of the class group $\Cl_{K}$ of a random imaginary quadratic field extension $K$ of $\bQ$ with a probability inversely proportional to $|\Aut_{\bZ}(H)|$.

\

\begin{conj}[Cohen-Lenstra]\label{CLconj} Given notations above, we must have

$$\lim_{N \ra \infty}\Prob_{K \in \IQ_{\leq N}}(\Cl_{K}[p^{\infty}] \simeq H) = \frac{1}{|\Aut_{\bZ}(H)|}\prod_{i=1}^{\infty}(1 - p^{-i}),$$

\

where $\IQ_{\leq N}$ is the set of imaginary quadratic fields over $\bQ$ whose absolute discriminant is $\leq N$ and the probability is given uniformly at random in this set. 
\end{conj}

\

\hspace{3mm} Let $n \in \bZ_{\geq 1}$ be the size of a finite set $S$ of some maximal ideals of $\mc{O}_{K}$, the ring of integers of $K$, that generate $\Cl_{K}$, as it is a finite abelian group. Then considering the exact sequence

$$\mc{O}_{K}^{S,\times} \ra \mf{I}_{K}^{S} \ra \Cl_{K} \ra 0,$$

\

where $\mc{O}_{K}^{S,\times} := \{x \in K^{\times} : x\mc{O}_{K} \text{ can be written as a product of positive/negative powers of ideals in }S\}$ and $\mf{I}^{S}_{K}$ is the abelian group of fractional ideals that can be written as a product of positive/negative powers of ideals in $S$, the fact that $\mc{O}_{K}^{S, \times}$ is a finitely generated abelian group of rank $n = |S|$ (because $K$ is imaginary quadratic) lets us have the following exact sequence:

$$\bZ^{n} \ra \bZ^{n} \ra \Cl_{K} \ra 0.$$

\

Applying $(-) \otimes_{\bZ} \bZ_{p}$, we have the exact sequence

\[\bZ_{p}^{n} \ra \bZ_{p}^{n} \ra \Cl_{K}[p^{\infty}] \ra 0,\]

\

so a heuristic approach to examine Conjecture \ref{CLconj} is to compute the cokernel of a ``random'' $\bZ_{p}$-linear map $\bZ_{p}^{n} \ra \bZ_{p}^{n}$. Friedman and Washington (Proposition 1 in \cite{FW}) proved that 

\begin{align}\label{heu1}
\lim_{n \ra \infty}\Prob_{A \in \Mat_{n}(\bZ_{p})}(\coker(A) \simeq H) = \frac{1}{|\Aut_{\bZ}(H)|} \prod_{i=1}^{\infty}( 1 - p^{-i} ),
\end{align}

\

where the probability measure on $\Mat_{n}(\bZ_{p})$ is given by the Haar measure with total measure $1$.

\

\begin{rmk} We learned the above exposition from a talk given by Wood \cite{Wood}.
\end{rmk}

\

\hspace{3mm} Next, we briefly review another heuristic due to Friedman and Washington regarding an analogous statement to Conjecture \ref{CLconj} replacing $\bQ$ with $\bF_{q}(t)$. Note that any quadratic extension $K$ of $\bQ$ can be written as the form $K = \bQ(\sqrt{d})$ for some square-free integer $d$. The extension $K$ of $\bQ$ is imaginary if and only if $d < 0$, and this is equivalent to requiring that it has one place above infinity. In the case of dealing with a quadratic extension $K$ of $\bF_{q}(t)$, we assume $q$ is odd and restrict to the case $K = \bF_{q}(t)(\sqrt{d(t)})$ for some square-free $d(t) \in \bF_{q}[t]$ of degree $2g + 1$ with $g \in \bZ_{\geq 1}$. In this case, the corresponding smooth, projective, and geometrically irreducible curve $C_{K}$ over $\bF_{q}$ to $K$ has genus $g$. As a double cover over $\bP^{1}_{\bF_{q}}$, the curve $C_{K}$ has one $\bF_{q}$-point $\mf{p}$ above $\infty = [0:1] \in \bP^{1}(\bF_{q})$. This implies that we have an isomorphism $\Cl_{K} \simeq \Pic^{0}(C_{K})$, given by $[D] \mapsto [D] - \deg(D) [\mf{p}]$, where $\Pic^{0}(C_{K})$ is the abelian group of the degree $0$ divisor classes on $C_{K}$. Friedman and Washington (Section 5 of \cite{FW}) explained that the $p$-part of $\Pic^{0}(C_{K})$ occurs as the cokernel of $A - \id$ of the $p$-adic Tate module of $\Pic^{0}(C_{K} \times_{\bF_{q}} \ol{\bF_{q}})$, where $A$ is the auotomorphism of the Tate module induced by the Frobenius and $\id$ is the identity. The $p$-adic Tate module of $\Pic^{0}(C_{K} \times_{\bF_{q}} \ol{\bF_{q}})$ is known to be a free $\bZ_{p}$-module with rank $2g$, where $g$ is the genus of $C_{K}$ (e.g., p.34 of \cite{Mil}), so we have the exact sequence

$$\bZ_{p}^{2g} \xrightarrow{A - I_{2g}} \bZ_{p}^{2g} \ra \Cl_{K}[p^{\infty}] \ra 0,$$

\

where $A \in \GL_{2g}(\bZ_{p})$. As a supporting heuristic that these class groups also follow the same pattern to Conjecture \ref{CLconj}, Friedman and Washington (Section 4 of \cite{FW}) proved that

\begin{align}\label{heu2}
\lim_{n \ra \infty}\Prob_{A \in \GL_{n}(\bZ_{p})}(\coker(A - I_{n}) \simeq H) = \frac{1}{|\Aut_{\bZ}(H)|} \prod_{i=1}^{\infty}( 1 - p^{-i} ),
\end{align}

\

where the probability is given by restricting the Haar measure on $\Mat_{n}(\bZ_{p})$ to $\GL_{n}(\bZ_{p})$ and then normalizing it so that we get the total measure $1$.

\

\hspace{3mm} The two heuristic results (\ref{heu1}) and (\ref{heu2}) involve different mathematical objects in their motivations. The fact that these numerical results are the same was refereed as ``blurring'' by Friedman and Washington (Section 4 of \cite{FW}) for their own heuristic reason (Section 1 of \cite{FW}). In this paper, we show that these two results are consequences of a more general phenomenon. For instance, Theorem \ref{main3x} (with $R = \bZ_{p}$) states that given any monic polynomials $P_{1}(t), \dots, P_{r}(t) \in \bZ_{p}[t]$ such that the reduction modulo $p$ gives distinct irreducible polynomials in $\bF_{p}[t]$ and $\deg(P_{r}) = 1$, we have

\begin{align*}
\lim_{n \ra \infty}\Prob_{A \in \Mat_{n}(\bZ_{p})}\left(
\begin{array}{c}
\coker(P_{1}(A)) = \cdots = \coker(P_{r-1}(A)) = 0, \\
\coker(P_{r}(A)) \simeq H
\end{array}\right) = \frac{1}{|\Aut_{\bZ_{p}}(H)|} \prod_{j=1}^{r} \prod_{i=1}^{\infty}(1 - p^{-i\deg(P_{j})}).
\end{align*}

\

This immediately imply both results of Friedman and Washington. (See Corollary \ref{FW2} and its proof.) Moreover, our theorem contains more than (\ref{heu1}) and (\ref{heu2}). For instance, if $p$ is chosen that $-1$ is not a square in $\bF_{p}$, the above implies that

\begin{align*}
\lim_{n \ra \infty}\Prob_{A \in \GL_{n}(\bZ_{p})}\left(
\begin{array}{c}
\coker(A^{2} + I_{n}) = 0, \\
\coker(A - I_{n}) \simeq H
\end{array}\right) = \frac{1}{|\Aut_{\bZ}(H)|} \left(\prod_{i=1}^{\infty}( 1 - p^{-i} ) \right) \left(\prod_{j=1}^{\infty}( 1 - p^{-2j} ) \right).
\end{align*}

\

\begin{rmk} Despite the sounding heuristic, Conjecture \ref{CLconj} is notorious for its difficulty, and it is wide open except for the case $p = 3$. (Some progress for $p = 3$ in terms of ``surjection moment'' method due to Davenport and Heilbronn is explained in Section 8.5 of \cite{EVW}.) On the other hand, there has been a quantitative breakthrough for an analogous statement replacing $\bQ$ with $\bF_{q}(t)$ (for large $g$ and $q$ such that $q \not\equiv 0, 1 \mod p$) due to Ellenberg, Venkatesh, and Westerland (Theorem 1.2 of \cite{EVW}), using more geometric methods. Our work is not directly related to proving Conjecture \ref{CLconj}, but it connects different results used as heuristic evidence for the conjecture.
\end{rmk}

\

\hspace{3mm} It is interesting that the above result resembles the distribution given by (\ref{heu1}) and (\ref{heu2}) on the set of finite abelian $p$-groups, called the \textbf{Cohen-Lenstra distribution} (e.g., Section 8.1 of \cite{EVW}), and this motivates a more general definition of the Cohen-Lenstra distribution, which we will discuss in Section \ref{main}. This computation is also in accordance with the philosophy of ``universality'' described by Wood \cite{Woo19}, which essentially states that the distributions we construct with random matrices tend to follow the Cohen-Lenstra distribution and their variants. Indeed, Wood dealt with various probability measures on $\Mat_{n}(\bZ_{p})$ extensively generalizing the Haar measure case and showed that, asymptotically in $n$, the cokernel of a random $A \in \Mat_{n}(\bZ_{p})$ with respect to these measures follow the Cohen-Lenstra distribution (Theorem 1.2 of \cite{Woo19}). Our paper will stick with the Haar measure and its pushforwards given by the polynomial maps $P_{1}, \dots, P_{r} : \Mat_{n}(\bZ_{p}) \ra \Mat_{n}(\bZ_{p})$.

\

\hspace{3mm} In the next section, we give an even more general conjecture (Conjecture \ref{conj}). Some of our main theorems are special cases of this conjecture. We separated the main theorems as Theorem \ref{main1x}, Theorem \ref{main2x}, and Theorem \ref{main3x}, because their proofs are different. Theorem \ref{main1x} and Theorem \ref{main2x} can be equivalently stated as statements about a random matrix in $\Mat_{n}(\bF_{q})$, with respect to the uniform distribution. These equivalent statements are given in Theorem \ref{main1} and Theorem \ref{main2}, and in Section \ref{CLphil}, we will see that they are related to another computational heuristic of Conjecture \ref{CLconj} due to Cohen and Lenstra \cite{CL}.

\

\subsection{Acknowledgment} We would like to thank our advisor Michael Zieve for various supports, including the financial supports for the relevant travelings through NSF grant DMS-1162181 for G. Cheong and DMS-1601844 for Y. Huang. We thank Karen Smith and the University of Michigan, for nominating and granting Rackham one-term dissertation fellowship to both of us. Y. Huang thanks Professor Emeritus Gopal Prasad for the Prasad Family Fellowship. G. Cheong was supported by the University of Michigan and the University of California--Irvine, for a traveling relevant to this work. We thank Sasha Barvinok, Kwun Chung, Ofir Gorodetsky, Haoyang Guo, Nathan Kaplan, Hendrik Lenstra, Eric Rains, and Melanie Matchett Wood for helpful conversations regarding this paper. We thank Jordan Ellenberg to bring our attention to \cite{Woo19}. We also thank Jason Fulman, Yuan Liu, and Brad Rodgers for helpful conversations regarding the previous versions of this paper. Finally, we thank Zhan Jiang for helping us initiate this project.

\

\section{Main conjecture and theorems}\label{main}

\hspace{3mm} Instead of $\bZ_{p}$, we will work more generally with any complete DVR (discrete valuation ring) $R$ with the maximal ideal $\mf{m}$, or simply denoted as $(R, \mf{m})$, whose residue field $R/\mf{m}$ is finite so that we may write $R/\mf{m} = \bF_{q}$. For any such $R$, saying that an $R$-module has finite size is equivalent to saying that it is of finite length. Finite abelian $p$-groups are finite size $\bZ_{p}$-modules, so they are finite length $\bZ_{p}$-modules. The following statement with $R = \bZ_{p}$ was given as Proposition 1 of \cite{FW}, and the proof given there works for general $R$.

\

\begin{prop}[Friedman-Washington]\label{FW} Let $(R, \mf{m})$ be a complete DVR with $R/\mf{m} = \bF_{q}$. Given any finite length $R$-module $H$, we have

$$\Prob_{A \in \Mat_{n}(R)}(\coker(A) \simeq H) = \left\{
	\begin{array}{ll}
	|\Aut_{R}(H)|^{-1} \left[ \prod_{i=1}^{n}( 1 - q^{-i} ) \right]\left[ \prod_{j=n-l_{H}+1}^{n} ( 1 - q^{-j} ) \right] & \mbox{if } n \geq l_{H},  \\
	0 & \mbox{if } n < l_{H},
	\end{array}\right.$$

\

where $l_{H} := \dim_{\bF_{q}}(H/\mf{m}H)$. In particular, we have

$$\lim_{n \ra \infty}\Prob_{A \in \Mat_{n}(R)}(\coker(A) \simeq H) = 
	\frac{1}{|\Aut_{R}(H)|} \prod_{i=1}^{\infty}( 1 - q^{-i} ).$$
\end{prop}

\

\hspace{3mm} Our paper generalizes the limiting distribution (i.e., the probability when $n$ goes to infinity) in Proposition \ref{FW} as Theorem \ref{main3x}. We also propose a more general conjecture in Conjecture \ref{conj} and solve more cases of it as Theorem \ref{main1x} and Theorem \ref{main2x}, which will be related to Conjecture \ref{CLconj} in a different way.

\

\subsection*{Notations} Given any ring $R$, we denote by $\Mod_{R}^{< \infty}$ the set of isomorphism classes of finite size $R$-modules. When $(R, \mf{m})$ is a DVR with $R/\mf{m} = \bF_{q}$, this is the same as the set of isomorphism classes of finite length $R$-modules. When denoting an isomorphism class, we will interchangeably write a representative of it to denote the class.

\

\begin{rmk} It turns out that for any DVR $(R, \mf{m})$ with $R/\mf{m} = \bF_{q}$, the assignment 

$$\{H\} \mapsto \frac{1}{|\Aut_{R}(H)|} \prod_{i=1}^{\infty}( 1 - q^{-i} )$$

\

defines a probability measure on the finest $\sigma$-algebra on $\Mod_{R}^{< \infty}$ (e.g., Remark \ref{sum=1}). We call this the \textbf{Cohen-Lenstra distribution of $R$}, although the terminology is mostly used for the case $R = \bZ_{p}$ in literature (e.g., Section 8 of \cite{EVW}). Since $R$ is a PID (principal ideal domain), for any finite length $R$-module $H$, we have a unique partition $\ld = (\ld_{1}, \dots, \ld_{l})$ such that

$$H \simeq R/\mf{m}^{\ld_{1}} \op \cdots \op R/\mf{m}^{\ld_{l}}.$$

\

In this case, we will write $\ld(H) := \ld$. A result of Macdonald ((1.6) on p.181 of \cite{Mac}) states that the number $|\Aut_{R}(H)|$ only depends on $q = |R/\mf{m}|$ and $\ld$ so that we may write $w(q, \ld) = |\Aut_{R}(H)|$. Using this and Lemma \ref{key} with $y = 1$, one may check that

$$\ld \mapsto \frac{1}{w(q, \ld)} \prod_{i=1}^{\infty}( 1 - q^{-i} )$$

\

defines a probability distribution on the set $\mc{P}$ of partitions of nonnegative integers. We will not name this more general distribution because it will only appear in our conjecture, not in any of our theorems, but we think that Cohen and Lenstra were aware of these distributions given the context of \cite{CL}. Fulman and Kaplan \cite{FK} discussed other similar distributions defined on $\mc{P}$ that come up in various combinatorial contexts.
\end{rmk}

\

\subsection{Main conjecture and theorems} We first introduce our main conjecture about a random matrix $A \in \Mat_{n}(R)$, where $(R, \mf{m})$ is a complete DVR such that $R/\mf{m} = \bF_{q}$. We will resolve special cases of this conjecture as Theorem \ref{main2x} and Theorem \ref{main3x} by understanding interplays between random matrices $A \in \Mat_{n}(R)$ and $\ol{A} \in \Mat_{n}(\bF_{q})$, where the latter is given by the uniform distribution on $\Mat_{n}(\bF_{q})$.

\

\begin{conj}\label{conj} Let $(R, \mf{m})$ be a complete DVR such that $R/\mf{m} = \bF_{q}$ and $P_{1}(t), \dots, P_{r}(t) \in R[t]$ monic polynomials such that the reduction modulo $\mf{m}$ gives distinct irreducible polynomials $\ol{P}_{1}(t), \dots, \ol{P}_{r}(t) \in \bF_{q}[t]$. Fix any $R$-modules $H_{1}, \dots, H_{r}$ of finite length. We must have

$$\lim_{n \ra \infty}\Prob_{A \in \Mat_{n}(R)}\left(\begin{array}{c}
\coker(P_{j}(A)) \simeq H_{j} \\
\text{ for } 1 \leq j \leq r
\end{array}\right) = \prod_{j=1}^{r} \frac{1}{w(q^{\deg(P_{j})}, \ld(H_{j}))} \prod_{i=1}^{\infty}(1 - q^{-i\deg(P_{j})}).$$
\end{conj}

\

\hspace{3mm} Note that the limiting distribution $n \ra \infty$ given by Proposition \ref{FW} is a special case of Conjecture \ref{conj}. More cases of Conjecture \ref{conj} are proven as Theorem \ref{main2x} and Theorem \ref{main3x}. Our main theorems are Theorem \ref{main1x}, Theorem \ref{main2x}, and Theorem \ref{main3x}.

\

\begin{thmx}\label{main1x} Let $(R, \mf{m})$ be a complete DVR such that $R/\mf{m} = \bF_{q}$ and $P(t) \in R[t]$ a monic polynomial such that the reduction modulo $\mf{m}$ gives an irreducible polynomial $\ol{P}(t) \in \bF_{q}[t]$. We have

$$\Prob_{A \in \Mat_{n}(R)}(
\coker(P(A)) = 0) = b_{n}(\deg(P)) \prod_{i=1}^{n}(1 - q^{-i}),$$

\

where $b_{n}(d)$, for $d \in \bZ_{\geq 0}$, are given by

$$\sum_{n=0}^{\infty}b_{n}(d) u^{n} = \prod_{i=1}^{\infty}\frac{1 - (q^{-i}u)^{d}}{1 - q^{1-i}u} \in \bC \llb u \rrb.$$

\

Moreover, we have

$$\lim_{n \ra \infty}b_{n}(d) = \prod_{i=1}^{\infty}\frac{1 - q^{-id}}{1 - q^{-i}},$$

\

so in particular, we have

$$\lim_{n \ra \infty}\Prob_{A \in \Mat_{n}(R)}(
\coker(P(A)) = 0) = \prod_{i=1}^{\infty}(1 - q^{-i\deg(P)}).$$
\end{thmx}

\

\begin{rmk} It will turn out that $b_{n}(d)$ given above are positive rational numbers explicitly given as

$$b_{n}(d) = \frac{|\{\ol{A} \in \Mat_{n}(\bF_{q}) : \coker(\ol{P}(\ol{A})) = 0\}|}{|\GL_{n}(\bF_{q})|},$$

\

for any degree $d$ monic irreducible polynomial $\ol{P}(t) \in \bF_{q}[t]$. This will appear in the proof of Theorem \ref{main1}, which is a step to prove Theorem \ref{main1x}. To check why $b_{n}(d)$ ought to be given this way, apply Lemma \ref{Haar} with $N = 0$ and $r = 1$ to the statement of Theorem \ref{main1x}.
\end{rmk}

\

\begin{thmx}\label{main2x}  Let $(R, \mf{m})$ be a complete DVR such that $R/\mf{m} = \bF_{q}$ and $P_{1}(t), \dots, P_{r}(t) \in R[t]$ monic polynomials such that the reduction modulo $\mf{m}$ gives distinct irreducible polynomials $\ol{P}_{1}(t), \dots, \ol{P}_{r}(t) \in \bF_{q}[t]$. We have

$$\lim_{n \ra \infty}\Prob_{A \in \Mat_{n}(R)}\left(\begin{array}{c}
\coker(P_{j}(A)) = 0 \\
\text{ for } 1 \leq j \leq r
\end{array}\right) = \prod_{j=1}^{r} \prod_{i=1}^{\infty}(1 - q^{-i\deg(P_{j})}).$$
\end{thmx}

\

\hspace{3mm} That is, Theorem \ref{main2x} generalizes the limiting result in Theorem \ref{main1x} by saying that for $1 \leq i < j \leq r$, the event that $\coker(P_{i}(A))$ vanishes is asymptotically independent to the event that $\coker(P_{j}(A))$ vanishes. This is surprising because specifying $P_{i}(A)$ and $P_{j}(A)$ are dependent (e.g., take $\deg(P_{i}) = \deg(P_{j}) = 1$), but somehow taking cokernels introduce independence. Our last theorem, introduced in the introduction for the specific case $R = \bZ_{p}$, has a similar feature (and so does Conjecture \ref{conj}).

\

\begin{thmx}\label{main3x} Let $(R, \mf{m})$ be a complete DVR such that $R/\mf{m} = \bF_{q}$ and $P_{1}(t), \dots, P_{r}(t) \in R[t]$ monic polynomials such that the reduction modulo $\mf{m}$ gives distinct irreducible polynomials $\ol{P}_{1}(t), \dots, \ol{P}_{r}(t) \in \bF_{q}[t]$. Suppose that $r \geq 1$ and $\deg(P_{r}) = 1$. Given any $R$-module $H$ of finite length, we have

\begin{align*}
\lim_{n \ra \infty}\Prob_{A \in \Mat_{n}(R)}\left(\begin{array}{c}
\coker(P_{1}(A)) = \cdots = \coker(P_{r-1}(A)) = 0 \\
\text{and } \coker(P_{r}(A)) \simeq H
\end{array}\right) = \frac{1}{|\Aut_{R}(H)|} \prod_{j=1}^{r} \prod_{i=1}^{\infty}(1 - q^{-i\deg(P_{j})}).
\end{align*}
\end{thmx}

\

\hspace{3mm} Note that Theorem \ref{main3x} generalizes the limiting distribution given in Proposition \ref{FW}, a result of Friedman and Washington. Theorem \ref{main3x} also generalizes another result of the same authors ((9) on p.234 in \cite{FW}), as we mentioned in the introduction.

\

\begin{cor}[Friedman and Washington]\label{FW2} Let $(R, \mf{m})$ be any complete DVR with $R/\mf{m} = \bF_{q}$ and $H$ any $R$-module of finite length. We have

$$\lim_{n \ra \infty}\Prob_{A \in \GL_{n}(R)}(\coker(A - I_{n}) \simeq H) = \frac{1}{|\Aut_{R}(H)|}\prod_{i=1}^{\infty}(1 - q^{-i}).$$
\end{cor}

\begin{proof} Choose any $N \geq 1$ such that $\mf{m}^{N}H = 0$. Since

$$\frac{|\GL_{n}(R/\mf{m}^{N+1})|}{|\Mat_{n}(R/\mf{m}^{N+1})|} = \frac{|\GL_{n}(\bF_{q})|}{|\Mat_{n}(\bF_{q})|} = \prod_{i=1}^{n}(1 - q^{-i}),$$

\

we have

\begin{align*}
\Prob_{\ol{A} \in \Mat_{n}(R/\mf{m}^{N+1})}\left(
\begin{array}{c}
\coker(\ol{A}) = 0, \\
\coker(\ol{A} - I_{n}) \simeq H
\end{array}
\right) &= \frac{|\GL_{n}(R/\mf{m}^{N+1})|}{|\Mat_{n}(R/\mf{m}^{N+1})|}\Prob_{\ol{A} \in \GL_{n}(R/\mf{m}^{N+1})}(\coker(\ol{A} - I_{n}) \simeq H) \\
&= \Prob_{A \in \GL_{n}(R)}(\coker(A - I_{n}) \simeq H)\prod_{i=1}^{n}(1 - q^{-i}),
\end{align*}

\

so applying Lemma \ref{Haar} and Theorem \ref{main3x} with $P_{1}(t) = t$ and $P_{2}(t) = t - 1$ for $r = 2$, we obtain the result by letting $n \ra \infty$.
\end{proof}

\

\begin{rmk} It seems that Theorem \ref{main3x} is new even for the case $R = \bZ_{p}$. Our proof for Theorem \ref{main3x} uses Lemma \ref{count} due to Friedman and Washington, which appears in the original proof of Corollary \ref{FW2}. In fact, our proof will show more generally that given the same hypothesis as in Theorem \ref{main3x}, we have

\begin{align*}
&\Prob_{A \in \Mat_{n}(R)}\left(\begin{array}{c}
\coker(P_{1}(A)) = \cdots = \coker(P_{r-1}(A)) = 0 \\
\text{and } \coker(P_{r}(A)) \simeq H
\end{array}\right) \\
&= \frac{q^{l_{H}^{2}}\prod_{i=1}^{l_{H}}( 1 - q^{-i} )^{2}}{ |\Aut_{R}(H)| } \Prob_{\ol{A} \in \Mat_{n}(\bF_{q})}\left(
\begin{array}{c}
\coker(P_{j}(\ol{A})) = 0 \text{ for } 1 \leq j \leq r-1, \\
\dim_{\bF_{q}}(\coker(P_{r}(\ol{A})) = l_{H}
\end{array}
\right),
\end{align*}

\

where $l_{H} = \dim_{\bF_{q}}(H/\mf{m}H)$. By taking $r = 1$ and $P_{1}(t) = t$ and using the fact that the number of matrices in $\Mat_{n}(\bF_{q})$ with corank $0 \leq l \leq n$ is equal to

$$\frac{q^{n^{2}-l^{2}}\prod_{i=l+1}^{n}( 1 - q^{-i} )^{2}}{\prod_{j=1}^{n-l}( 1 - q^{-j} )},$$

\

we can deduce Proposition \ref{FW} even for all $n \geq 0$, not just $n \ra \infty$. This is not the proof given by Friedman and Washington \cite{FW} (as one can check Proposition 1 in their paper). However, Lemma \ref{count} is from their paper, and it is quite evident that Friedman and Washington were aware of this argument.
\end{rmk}

\

\begin{rmk} Given our discussion, the known cases for Conjecture \ref{conj} to our best knowledge are the following:

\begin{itemize}
	\item any $r \geq 0$ with $H_{1} = \cdots = H_{r} = 0$ (Theorem \ref{main2x});
	\item any $r \geq 1$ with $\deg(P_{r}) = 1$ while $H_{1} = \cdots H_{r-1} = 0$ and any $H_{r}$ (Theorem \ref{main3x}).
\end{itemize}
\end{rmk}

\

\subsection{Random matrices over finite fields}\label{F_q} Among our three theorems, Theorem \ref{main1x} and Theorem \ref{main2x} can be rephrased as statements about $\ol{A} \in \Mat_{n}(\bF_{q})$, chosen uniformly at random. In this section, we will write $A$ instead of $\ol{A}$ for convenience. Theorem \ref{main1x} will be deduced from the following.

\

\begin{thm}\label{main1} Fix any monic irreducible polynomial $P = P(t) \in \bF_{q}[t]$ and a $P^{\infty}$-torsion $\bF_{q}[t]$-module $H$ of finite length. Write $h := \dim_{\bF_{q}}(H)$. Then

\begin{align*}
\Prob_{A \in \Mat_{n}(\bF_{q})}(A[P^{\infty}] \simeq H)
&= \left\{
	\begin{array}{ll}
	\frac{b_{n-h}(\deg(P))}{|\Aut_{\bF_{q}[t]}(H)|}\prod_{i=1}^{n}(1 - q^{-i}) & \mbox{if } n \geq h \text{ and } \\
	0 & \mbox{if } n < h,
	\end{array}\right.
\end{align*}

\

where $b_{n}(d)$, for $d \in \bZ_{\geq 0}$, are given by

$$\sum_{n=0}^{\infty}b_{n}(d) u^{n} = \prod_{i=1}^{\infty}\frac{1 - (q^{-i}u)^{d}}{1 - q^{1-i}u} \in \bC \llb u \rrb.$$

\

Moreover, we have

$$\lim_{n \ra \infty}b_{n}(d) = \prod_{i=1}^{\infty}\frac{1 - q^{-id}}{1 - q^{-i}}$$

\

so that

$$\lim_{n \ra \infty}\Prob_{A \in \Mat_{n}(\bF_{q})}(A[P^{\infty}] \simeq H) = \frac{1}{|\Aut_{\bF_{q}[t]}(H)|}\prod_{i=1}^{\infty}(1 - q^{-i\deg(P)}).$$
\end{thm}

\

\begin{rmk} Note that given $q, n$, and $H$, the conclusion of Theorem \ref{main1} only depends on $\deg(P)$. A special case where $\deg(P) = 1$ is interesting (i.e., $P(t) = t - a$ for some $a \in \bF_{q}$). Since $b_{n}(1) = 1$ for all $n \in \bZ_{\geq 0}$, Theorem \ref{main1} implies that

\begin{align*}
\Prob_{A \in \Mat_{n}(\bF_{q})}(A[(t-a)^{\infty}] \simeq H)
&= \left\{
	\begin{array}{ll}
	\frac{1}{|\Aut_{\bF_{q}[t]}(H)|}\prod_{i=1}^{n}(1 - q^{-i}) & \mbox{if } n \geq \dim_{\bF_{q}}(H) \text{ and } \\
	0 & \mbox{if } n < \dim_{\bF_{q}}(H).
	\end{array}\right.
\end{align*}
\end{rmk}

\

\hspace{3mm} Likewise, Theorem \ref{main2x} will be deduced from the following.

\

\begin{thm}\label{main2} Fix any distinct monic irreducible polynomials $P_{1}(t), \dots, P_{r}(t) \in \bF_{q}[t]$ and $P_{j}^{\infty}$-torsion module $H_{j}$ of finite length for $1 \leq j \leq r$. Then

\begin{align*}
\lim_{n \ra \infty}\Prob_{A \in \Mat_{n}(\bF_{q})}(A[P_{j}^{\infty}] \simeq H_{j} \text{ for } 1 \leq j \leq r) &=	\prod_{j=1}^{r}\frac{1}{|\Aut_{\bF_{q}[t]}(H_{j})|}\prod_{i=1}^{\infty}(1 - q^{-i\deg(P_{j})}).
\end{align*}
\end{thm}

\

\hspace{3mm} As an immediate corollary, we see how random matrices in $\GL_{n}(\bF_{q})$ is related to Cohen-Lenstra distributions as $n \ra \infty$. This is originally due to Fulman in his thesis \cite{Ful97}, but a partial result to this was also observed by Washington prior to Fulman (Theorem 1 (b) in \cite{Was}). Washington's result can be obtained by taking $P(t) = t - 1$ in the following corollary and applying Lemma \ref{CL2}, which is due to Cohen and Lenstra.

\

\begin{cor}[cf. \cite{Ful14}]\label{Ful} Fix any monic irreducible polynomial $P(t) \in \bF_{q}[t] \sm \{t\}$ and a $P^{\infty}$-torsion $\bF_{q}[t]$-module $H$ of finite length. Then

\begin{align*}
\lim_{n \ra \infty}\Prob_{A \in \GL_{n}(\bF_{q})}(A[P^{\infty}] \simeq H) &=	\frac{1}{|\Aut_{\bF_{q}[t]}(H)|}\prod_{i=1}^{\infty}(1 - q^{-i\deg(P)}).
\end{align*}
\end{cor}

\begin{proof} Applying Theorem \ref{main2} by taking $P_{1}(t) = t$ and $P_{2}(t) = P(t)$ with $H_{1} = 0$ and $H_{2} = H$, we get

\begin{align*}
\lim_{n \ra \infty} \Prob_{A \in \GL_{n}(\bF_{q})}(A[P^{\infty}] \simeq H) &= \lim_{n \ra \infty}\frac{|\{A \in \GL_{n}(\bF_{q}) : A[P^{\infty}] \simeq H\}|}{|\Mat_{n}(\bF_{q})|}\frac{|\Mat_{n}(\bF_{q})|}{|\GL_{n}(\bF_{q})|} \\
&= \frac{\lim_{n \ra \infty} \Prob_{A \in \Mat_{n}(\bF_{q})}(A[t^{\infty}] = 0 \text{ and } A[P^{\infty}] \simeq H)}{\prod_{i=1}^{\infty}(1 - q^{-i})} \\
&= \frac{1}{|\Aut_{\bF_{q}[t]}(H)|}\prod_{i=1}^{\infty}\frac{(1 - q^{-i})(1 - q^{-i\deg(P)})}{(1 - q^{-i})} \\
&= \frac{1}{|\Aut_{\bF_{q}[t]}(H)|}\prod_{i=1}^{\infty}(1 - q^{-i\deg(P)}),
\end{align*}

\

as desired.
\end{proof}

\

\begin{rmk} Thanks to Nathan Kaplan, we have noticed that Boreico has independently obtained Theorem \ref{main1} in his thesis (Theorem 3.8.18 in \cite{Bor}) prior to our paper. Boreico's proof is different from ours, but he also sketches our proof and discusses the same corollary (i.e., Corollary \ref{Ful}). We believe that providing our proof for Theorem \ref{main1} is still valuable for clarity and details. To our best knowledge, Boreico's thesis was never published nor made into a preprint, but we recommend the interested reader take a look at his alternative proof of Theorem \ref{main1} (i.e., Theorem 3.8.18 in \cite{Bor}) which uses more direct linear algebraic and measure theoretic arguments. Boreico's proof also inspired us to find many connections between our results over $\bF_{q}$ and random matrices over an arbitrary complete DVR whose residue field at its maximal ideal is $\bF_{q}$. A part of his proof is presented in this paper as Lemma \ref{Bor}. We use this to get Corollary \ref{formula}, which will enable us to see that Theorem \ref{main1} and Theorem \ref{main2} conversely imply Theorem \ref{main1x} and Theorem \ref{main2x} as well.
\end{rmk}

\

\section{Philosophy of Cohen and Lenstra}\label{CLphil}

\subsection*{Notations} Given a ring $R$ and integer $N \geq 1$, we denote by $\Mod_{R}^{\leq N}$ the set of isomorphism classes of $R$-modules whose size is less than equal to $N$ and $\Mod_{R}^{=N}$ the set of isomorphism classes of $R$-modules whose size is equal to $N$. 

\

\hspace{3mm} Conjecture \ref{CLconj} was motivated by the numerical observation of Cohen and Lenstra that most class groups of imaginary quadratic field extension of $\bQ$ is cyclic and that ``the scarcity of noncylic groups can be attributed to the fact that they have many automorphisms'' (as in the first page of \cite{CL}). For instance, note that

\[|\Aut_{\bZ}(\bZ/(5) \op \bZ/(5))| = 480,\]

\

while

\[|\Aut_{\bZ}(\bZ/(25))| = 20,\]

\

even though the groups $\bZ/(5) \op \bZ/(5)$ and $\bZ/(25)$ have the same size. Hence, if this speculation is true, for $N \gg 0$, the the probability we choose $\bZ/(25)$ from $\IQ_{\leq N}$ (as in the introduction) uniformly at random should be about $480/20 = 24$ times larger than the probability we choose $\bZ/(5) \op \bZ/(5)$ similarly. Cohen and Lesntra made a hypothesis that the limiting distribution in $N$ of the class group of a random $K \in \IQ_{\leq N}$ would be similar to that of a random finite abelian group $A$, whose probability of occurrence is proportional to $1/|\Aut_{\bZ}(A)|$. They showed that for any finite abelian $p$-group $H$, we have

$$\lim_{N \ra \infty}\Prob_{A \in \Mod_{\bZ}^{\leq N}}(A[p^{\infty}] \simeq H) = \frac{1}{|\Aut_{\bZ}(H)|}\prod_{i=1}^{\infty} (1 - p^{-i}),$$

\

where we used the following definition with $S = \Mod_{\bZ}^{\leq N}$:

\begin{defn} Given a nonempty finite subset $S$ of the isomorphism classes of a category $\mc{C}$, all of whose automorphism groups are finite, we define

$$\Prob_{s \in S}(s \text{ satisfies } \ms{P}) := \frac{\sum_{\substack{s \in S, \\ s \text{ satisfies } \ms{P}}} 1/|\Aut_{\mc{C}}(s)|}{\sum_{s \in S} 1/|\Aut_{\mc{C}}(s)|},$$

\

where $\ms{P}$ is any property on $S$.
\end{defn}

\

\hspace{3mm} This provides another heuristic philosophy behind Conjecture \ref{CLconj}, which historically predates Proposition \ref{FW}. The statistics on $\Mat_{n}(\bF_{q})$ has also much to do with this philosophy. Under the conjugate action $\GL_{n}(\bF_{q}) \acts \Mat_{n}(\bF_{q})$, the set $\Mat_{n}(\bF_{q})/\GL_{n}(\bF_{q})$ of orbits parametrizes the set $\Mod_{\bF_{q}[t]}^{=q^{n}}$ of the isomorphism classes of $\bF_{q}[t]$-modules of $\bF_{q}$-dimension $n$ because each matrix $\ol{A} \in \Mat_{n}(\bF_{q})$ gives $\bF_{q}^{n}$ an $\bF_{q}[t]$-module structure, which we denote as $\ol{A} \acts \bF_{q}^{n}$, by $t \cdot v := \ol{A}v$ for $v \in \bF_{q}^{n}$ and two matrices define isomorphic $\bF_{q}[t]$-module structures if and only if they are in the same orbit under the conjugate action of $\GL_{n}(\bF_{q})$. Noting that

$$\Aut_{\bF_{q}[t]}(\ol{A} \acts \bF_{q}^{n}) = \Stab_{\GL_{n}(\bF_{q})}(\ol{A}),$$

\

by an application of the orbit-stabilizer theorem, we have

$$\Prob_{\ol{A} \in \Mat_{n}(\bF_{q})}(\ol{A} \text{ satisfies } \ms{P}) = \Prob_{\ol{A} \in \Mod_{\bF_{q}[t]}^{=q^{n}}}(\ol{A} \text{ satisfies } \ms{P}).$$

\

Therefore, Theorem \ref{main1} and Theorem \ref{main2} can be reinterpreted as the computations on explicit probability distributions on $\Mod_{\bF_{q}[t]}^{=q^{n}}$. Cohen and Lenstra considered a similar distribution on $\Mod_{\bF_{q}[t]}^{\leq q^{n}}$ instead of $\Mod_{\bF_{q}[t]}^{=q^{n}}$ in Theorem \ref{main2}. Their proof works for many Dedekind domains $R$ including $\bZ$ and $\bF_{q}[t]$, but it requires that there are finitely many finite length $R$-modules $M$ with $|M| \leq N$ for any $N > 0$ (up to isomorphisms) and the zeta function $\zeta_{R}(s)$ must have only one simple pole at $s = 1$.

\

\begin{prop}[Example 5.9 in \cite{CL}, $u = 0$]\label{CL} Let $R$ be a number ring or the coordinate ring of the open subset obtained by a smooth, geometrically connected, and projective curve over $\bF_{q}$ minus an $\bF_{q}$-point. Fix finitely many maximal ideals $\mf{m}_{1}, \dots, \mf{m}_{r}$ of $R$. For $1 \leq j \leq r$, say $H_{j}$ is an $\mf{m}_{j}^{\infty}$-torsion $R$-module of finite length and $q_{j} := |R/\mf{m}_{j}|$. We have

$$\lim_{N \ra \infty}\Prob_{A \in \Mod_{R}^{\leq N}}\left(
\begin{array}{c}
A[\mf{m}_{j}^{\infty}] \simeq H_{j} \\
\text{for } 1 \leq j \leq r
\end{array}
\right) = \prod_{j=1}^{r}\frac{1}{|\Aut_{R}(H_{j})|}\prod_{i=1}^{\infty} (1 - q_{j}^{-i}).$$
\end{prop}

\

\begin{rmk} Roughly speaking, Theorem \ref{main2x}, Theorem \ref{main3x}, Theorem \ref{main2}, and Proposition \ref{CL} (for the case $R = \bF_{q}[t]$) tell us about how distributions involving some global information about $\bA^{1}_{\bF_{q}} = \Spec(\bF_{q}[t])$ can be obtained by their local information. As their invariants such as $n$ or $N$ go to infinity, their local events become independent. 
\end{rmk}

\

\begin{rmk} Continuing the proof of Corollary \ref{Ful}, we have

\begin{align*}\lim_{n \ra \infty}\Prob_{\ol{A} \in \GL_{n}(\bF_{q})}(\ol{A}[P^{\infty}] \simeq H) &= \lim_{n \ra \infty}\Prob_{A \in \Mod_{\bF_{q}[t]}^{= q^{n}}}\left(
\begin{array}{c}
A[t^{\infty}] = 0 \text{ and} \\\
A[P(t)^{\infty}] \simeq H
\end{array}
\right) \\
&= \lim_{n \ra \infty}\Prob_{A \in \Mod_{\bF_{q}[t]}^{\leq q^{n}}}\left(
\begin{array}{c}
A[t^{\infty}] = 0 \text{ and} \\\
A[P(t)^{\infty}] \simeq H
\end{array}
\right),
\end{align*}

\

so Fulman's result about random matrices in $\GL_{n}(\bF_{q})$ can be realized as a special case of Proposition \ref{CL}, a heuristic result due to Cohen-Lenstra, where they came up with Cohen-Lenstra distributions in the first place. This provides a concrete reason why a random matrix in $\GL_{n}(\bF_{q})$ produces a Cohen-Lenstra distribution (as $n \ra \infty$), resolving previous inquiries made by Washington \cite{Was}, Lengler \cite{Len}, Fulman \cite{Ful14}, and Fulman-Kaplan \cite{FK}. In general, many algebraic objects, whose probability of occurrence is inversely proportional to the numbers of their automorphisms, seem to follow some version of Cohen-Lenstra distribution, and our results exemplify such phenomena. More broad examples on ``universal'' occurrences of Cohen-Lenstra distributions (or similar looking distributions) can be found in literature (e.g., \cite{Woo17} and \cite{Woo19}), and this seems to be an active area of research.
\end{rmk}

\

\section{Converting Haar measure problems into problems over finite local rings}

\hspace{3mm} In this section, we explain how to reduce the problems of computing the probabilities in Theorem \ref{main1x}, Theorem \ref{main2x} and Theorem \ref{main3x} given by the Haar measure on $\Mat_{n}(R)$ into some combinatorial problems over finite local rings. This will be used in the next section when we show how Theorem \ref{main1} and Theorem \ref{main2} imply Theorem \ref{main1x} and Theorem \ref{main2x}.

\

\begin{lem}\label{red} Let $(R, \mf{m})$ be a complete DVR with $R/\mf{m} = \bF_{q}$ and $H$ a finite length $R$-module. Fix any $N \in \bZ_{\geq 0}$ such that $\mf{m}^{N} H = 0$, noting that there always exists such $N$. For any $A \in \Mat_{n}(R)$, we have $\coker(A) \simeq H$ if and only if $\coker(\ol{A}) \simeq H$, where $\ol{A} \in \Mat_{n}(R/\mf{m}^{N+1})$ is the image of $A$ modulo $\mf{m}^{N+1}$.
\end{lem}

\begin{proof} If $\coker(A) \simeq H$, then $\coker(\ol{A}) \simeq H/\mf{m}^{N+1}H \simeq H$ because $\mf{m}^{N+1}H = \mf{m}\mf{m}^{N}H = 0$. Conversely, let $\coker(\ol{A}) \simeq H$. Since $R$ is a PID, we may have

$$H \simeq R/\mf{m}^{\ld_{1}} \op \cdots \op R/\mf{m}^{\ld_{l}}$$

\

for some partition $\ld = (\ld_{1}, \dots, \ld_{l})$. Since $\mf{m}^{N}H = 0$, we have $\ld_{i} \leq N$ for all $i$. Choosing a generator $\pi$ of $\mf{m}$, the fact that $R$ is a PID lets us choose $g_{1}, g_{2} \in \GL_{n}(R)$ such that $g_{1}Ag_{2}$ is a diagonal matrix (so-called a \textbf{Smith normal form} of $A$). Since $R$ is a DVR, each diagonal entry of $g_{1}Ag_{2}$ is either $0$ or of the form $u \pi^{e}$, where $u$ is a unit of $R$ and $e \in \bZ_{\geq 0}$. There should not be any $0$ in the diagonal entries modulo $\mf{m}^{N+1}$ because $\coker(\ol{A}) \simeq H$ is annihilated by $\mf{m}^{N}$. (This is why our conclusion is about $A$ modulo $\mf{m}^{N+1}$ instead of $\mf{m}^{N}$.) Thus, the diagonal entries of $g_{1}Ag_{2}$ are of the form $u_{1}\pi^{e_{1}}, \dots, u_{n}\pi^{e_{n}}$, where $u_{i} \in R^{\times}$ and $0 \leq e_{i} \leq N$. The matrix $\ol{g_{1}} \ol{A} \ol{g_{2}} \in \Mat_{n}(R/\mf{m}^{N+1})$ is diagonal with nonzero entires $\ol{u_{1}}\ol{\pi}^{e_{1}}, \dots, \ol{u_{n}}\ol{\pi}^{e_{n}} \in R/\mf{m}^{N+1}$. We must have $(e_{1}, \dots, e_{n}) = (\ld_{1}, \dots, \ld_{l}, 0, \dots, 0)$ because $\ol{g_{1}}, \ol{g_{2}} \in \GL_{n}(R/\mf{m}^{N+1})$ so that 

\begin{align*}
R/\mf{m}^{e_{1}} \op \cdots \op R/\mf{m}^{e_{n}} &\simeq \coker(\ol{g_{1}}\ol{A}\ol{g_{2}}) \\
&\simeq \coker(\ol{A}) \\
&\simeq H \\
&\simeq R/\mf{m}^{\ld_{1}} \op \cdots \op R/\mf{m}^{\ld_{l}}.
\end{align*}

\

Therefore, we have

\begin{align*}
\coker(A) &\simeq R/\mf{m}^{e_{1}} \op \cdots \op R/\mf{m}^{e_{n}} \\
&\simeq R/\mf{m}^{\ld_{1}} \op \cdots \op R/\mf{m}^{\ld_{l}} \\
&\simeq H,
\end{align*}

as desired.
\end{proof}

\

\begin{rmk} The easiest case of Lemma \ref{red} is when $N = 0$, which necessarily means $H = 0$. For this case, the lemma can be proven by a direct application of Nakayama's lemma. This special case is all we need for Theorem \ref{main1x} and Theorem \ref{main2x}, but the full version of Lemma \ref{red} is needed for proving Theorem \ref{main3x}. We will not directly use Lemma \ref{red}, but it will be used to prove the following lemma, directly applicable for proving all of our main theorems. It describes how we may concretely think of certain events according to the Haar measure on $\Mat_{n}(R)$.
\end{rmk}

\

\begin{lem}\label{Haar} Let $(R, \mf{m})$ be a complete DVR with $R/\mf{m} = \bF_{q}$ and $H_{1}, \dots, H_{r}$ finite length $R$-modules so that we may pick some $N \in \bZ_{\geq 0}$ such that $\mf{m}^{N}H_{1} = \cdots = \mf{m}^{N}H_{r} = 0$. For any monic polynomials $f_{1}(t), \dots, f_{r}(t) \in R[t]$, we have

$$\Prob_{A \in \Mat_{n}(R)}\left(\begin{array}{c}
\coker(f_{j}(A)) \simeq H_{j} \\
\text{ for } 1 \leq j \leq r
\end{array}\right) = \Prob_{\ol{A} \in \Mat_{n}(R/\mf{m}^{N+1})}\left(\begin{array}{c}
\coker(f_{j}(\ol{A})) \simeq H_{j} \\
\text{ for } 1 \leq j \leq r
\end{array}\right).$$
\end{lem}

\begin{proof} Consider the projection $\Mat_{n}(R) \tra \Mat_{n}(R/\mf{m}^{N+1})$ given modulo $\mf{m}^{N+1}$. Denoting this map by $A \mapsto \ol{A}$, the Haar measure on $\Mat_{n}(R)$ assigns $1/|\Mat_{n}(R/\mf{m}^{N+1})|$ to the fiber $A + \mf{m}^{N+1}\Mat_{n}(R)$ of any $\ol{A} \in \Mat_{n}(R/\mf{m}^{N+1})$. Moreover, for any monic polynomial $f(t) \in R[t]$, a generator $\pi$ of $\mf{m}$, and any $B \in \Mat_{n}(R)$, we have $f(A + \pi^{N+1}B) = f(A) + \pi^{N+1}C$ for some $C \in \Mat_{n}(R)$. Thus, for any $R$-module $H$ with $\mf{m}^{N}H = 0$, we have $\coker(f(A)) \simeq H$ if and only if $\coker(f(A + \pi^{N+1}B)) \simeq H$ for all $B \in \Mat_{n}(R)$. Having this in mind, applying Lemma \ref{red} lets us see that

\begin{align*}
&\Prob_{A \in \Mat_{n}(R)}\left(\begin{array}{c}
\coker(f_{j}(A)) \simeq H_{j} \\
\text{ for } 1 \leq j \leq r
\end{array}\right) \\
&= \sum_{\ol{A} \in \Mat_{n}(R/\mf{m}^{N+1})} \mu_{n}
\left((A + \mf{m}^{N+1}\Mat_{n}(R)) \cap \left\{\begin{array}{c}
M \in \Mat_{n}(R) : \\
\coker(f_{j}(M)) \simeq H_{j} \text{ for } 1 \leq j \leq r
\end{array}\right\}\right) \\
&= \frac{1}{|\Mat_{n}(R/\mf{m}^{N+1})|} \left|\left\{\begin{array}{c}
\ol{A} \in \Mat_{n}(R/\mf{m}^{N+1}) : \\
\coker(f_{j}(\ol{A})) \simeq H_{j} \text{ for } 1 \leq j \leq r
\end{array}\right\}\right| \\
&= \Prob_{\ol{A} \in \Mat_{n}(R/\mf{m}^{N+1})}\left(\begin{array}{c}
\coker(f_{j}(\ol{A})) \simeq H_{j} \\
\text{ for } 1 \leq j \leq r
\end{array}\right),
\end{align*}

\

where $\mu_{n}$ denoted the Haar (probability) measure on $\Mat_{n}(R)$. This finishes the proof.
\end{proof}

\

\section{Reductions for Theorems \ref{main1x}, \ref{main2x}, \ref{main3x}}\label{reduction}

\subsection{Theorems \ref{main1} and \ref{main2} imply Theorems \ref{main1x} and Theorem \ref{main2x}} In this section, we show that Theorems \ref{main1} and \ref{main2} imply Theorems \ref{main1x} and Theorem \ref{main2x}, respectively.

\

\begin{proof}[Proof that Theorems \ref{main1} and \ref{main2} imply Theorems \ref{main1x} and Theorem \ref{main2x}] We keep the notations in Theorem \ref{main2x}. Taking $N = 0$ in Lemma \ref{Haar}, we have

$$\Prob_{A \in \Mat_{n}(R)}\left(\begin{array}{c}
\coker(P_{j}(A)) = 0 \\
\text{ for } 1 \leq j \leq r
\end{array}\right) = \Prob_{\ol{A} \in \Mat_{n}(\bF_{q})}\left(\begin{array}{c}
\coker(P_{j}(\ol{A})) = 0 \\
\text{ for } 1 \leq j \leq r
\end{array}\right).$$

\

Moreover, we note that for any $\ol{A} \in \Mat_{n}(\bF_{q})$, we have $\coker(P_{j}(\ol{A})) = 0$ if and only if $P_{j}(\ol{A}) = \ol{P}_{j}(\ol{A})$ is invertible in $\Mat_{n}(\bF_{q})$. This is the same as saying $A[\ol{P}_{j}^{\infty}] = 0$, so this finishes the proof by taking $H_{1} = \cdots = H_{r} = 0$ in Theorem \ref{main2} (and, taking $r = 1$, $H_{1} = H = 0$ in Theorem \ref{main1}).
\end{proof}

\

\begin{rmk} In the above proof, we only used the special cases of Theorem \ref{main1} and Theorem \ref{main2} when $H = 0$ and $H_{1} = \cdots = H_{r} = 0$ to deduce Theorem \ref{main1x} and Theorem \ref{main2x}, respectively. However, we will see with Corollary \ref{formula} that it is also easy to deduce Theorem \ref{main1} and Theorem \ref{main2} from Theorem \ref{main1x} and Theorem \ref{main2x}. Underlying this is a formula due to Boreico \cite{Bor} given as Lemma \ref{Bor}.
\end{rmk}

\

\subsection{Theorem \ref{main2} implies Theorem \ref{main3x}} This section is devoted for showing that Theorem \ref{main2} implies Theorem \ref{main3x}. The crucial lemma is the following result due to Friedman and Washington ($\#_{H}(\bar{R})$ on p.236 of \cite{FW}).

\

\begin{lem}\label{count} Let $(R, \mf{m})$ be a complete DVR with $R/\mf{m} = \bF_{q}$ and $H$ a finite length $R$-module. Choose any $N \in \bZ_{\geq 0}$ such that $\mf{m}^{N}H = 0$. Fix any monic polynomial $P(t) \in R[t]$ of degree $1$. For any $\ol{A} \in \Mat_{n}(\bF_{q})$, the number of lifts $A \in \Mat_{n}(R/\mf{m}^{N+1})$ of $\ol{A}$ such that $\coker(P(A)) \simeq H$ is equal to

$$\left\{
	\begin{array}{ll}
	q^{Nn^{2} +l_{H}^{2}} |\Aut_{R}(H)|^{-1} \prod_{i=1}^{l_{H}}( 1 - q^{-i} )^{2} & \mbox{if } \dim_{\bF_{q}}(\coker(P(\ol{A}))) = l_{H},  \\
	0 & \mbox{if } \dim_{\bF_{q}}(\coker(P(\ol{A}))) \neq l_{H},
	\end{array}\right.$$
\end{lem}

\

where $l_{H} := \dim_{\bF_{q}}(H/\mf{m}H)$.

\

\hspace{3mm} We will use another lemma due to Cohen and Lenstra (Theorem 6.3 in \cite{CL} with $u = 0$) as follows.

\

\begin{lem}[Cohen and Lenstra]\label{CL2} Let $(R, \mf{m})$ be a complete DVR with $R/\mf{m} = \bF_{q}$. For any $l \in \bZ_{\geq 0}$, we have

$$\Prob_{H \in \Mod_{R}^{<\infty}}(\dim_{\bF_{q}}(H/\mf{m}H) = l) = \frac{q^{-l^{2}}\prod_{i=1}^{\infty}(1 - q^{-i})}{\prod_{i=1}^{l}(1 - q^{-i})^{2}}$$

\

with respect to the Cohen-Lenstra distribution on $\Mod_{R}^{<\infty}$.
\end{lem}

\

\begin{proof}[Proof that Theorem \ref{main2} implies Theorem \ref{main3x}] Let $l_{H} := \dim_{\bF_{q}}(H/\mf{m}H)$ and choose $N \in \bZ_{\geq 0}$ such that $\mf{m}^{N}H = 0$. Similarly arguing as in the proof of Lemma \ref{Haar}, we can observe that the preimage of the set

$$\left\{
\begin{array}{c}
\ol{A} \in \Mat_{n}(\bF_{q}) : \\
\ol{A}[\ol{P}_{j}^{\infty}] = 0 \text{ for } 1 \leq j \leq r-1
\end{array}
\right\}
=
\left\{
\begin{array}{c}
\ol{A} \in \Mat_{n}(\bF_{q}) : \\
\coker(P_{j}(\ol{A})) = 0 \text{ for } 1 \leq j \leq r-1
\end{array}
\right\}$$

\

under the projection $\Mat_{n}(R/\mf{m}^{N+1}) \tra \Mat_{n}(\bF_{q})$ modulo $\mf{m}$ is precisely

$$\left\{
\begin{array}{c}
A \in \Mat_{n}(R/\mf{m}^{N+1}) : \\
\coker(P_{j}(A)) = 0 \text{ for } 1 \leq j \leq r-1
\end{array}
\right\}.$$

\

Since 

$$\dim_{\bF_{q}}(\coker(P_{r}(\ol{A}))) = \dim_{\bF_{q}}(\ker(P_{r}(\ol{A}))) = \dim_{\bF_{q}}(\ol{A}[\ol{P}_{r}^{\infty}]/\ol{P}_{r}\ol{A}[\ol{P}_{r}^{\infty}]),$$

\

applying Lemma \ref{count} implies that

\begin{align*}
\left|\left\{
\begin{array}{c}
A \in \Mat_{n}(R/\mf{m}^{N+1}) : \\
\coker(P_{j}(A)) = 0 \text{ for } 1 \leq j \leq r-1, \\
\coker(P_{r}(A)) \simeq H
\end{array}
\right\}\right| = \frac{q^{Nn^{2} +l_{H}^{2}} \prod_{i=1}^{l_{H}}( 1 - q^{-i} )^{2}}{ |\Aut_{R}(H)| } \left|\left\{
\begin{array}{c}
\ol{A} \in \Mat_{n}(\bF_{q}) : \\
\ol{A}[\ol{P}_{j}^{\infty}] = 0 \text{ for } 1 \leq j \leq r-1, \\
\dim_{\bF_{q}}(\ol{A}[\ol{P}_{r}^{\infty}]/\ol{P}_{r}\ol{A}[\ol{P}_{r}^{\infty}]) = l_{H}
\end{array}
\right\}\right|,
\end{align*}

\

so dividing by $q^{(N+1)n^{2}} = |\Mat_{n}(R/\mf{m}^{N+1})|$, we have

\begin{align*}
&\Prob_{A \in \Mat_{n}(R/\mf{m}^{N+1})}\left(
\begin{array}{c}
\coker(P_{j}(A)) = 0 \text{ for } 1 \leq j \leq r-1, \\
\coker(P_{r}(A)) \simeq H
\end{array}
\right) \\
&= \frac{q^{l_{H}^{2}}\prod_{i=1}^{l_{H}}( 1 - q^{-i} )^{2}}{ q^{n^{2}}|\Aut_{R}(H)| } \left|\left\{
\begin{array}{c}
\ol{A} \in \Mat_{n}(\bF_{q}) : \\
\ol{A}[\ol{P}_{j}^{\infty}] = 0 \text{ for } 1 \leq j \leq r-1, \\
\dim_{\bF_{q}}(\ol{A}[\ol{P}_{r}^{\infty}]/\ol{P}_{r}\ol{A}[\ol{P}_{r}^{\infty}]) = l_{H}
\end{array}
\right\}\right| \\
&= \frac{q^{l_{H}^{2}}\prod_{i=1}^{l_{H}}( 1 - q^{-i} )^{2}}{ |\Aut_{R}(H)| } \Prob_{\ol{A} \in \Mat_{n}(\bF_{q})}\left(
\begin{array}{c}
\ol{A}[\ol{P}_{j}^{\infty}] = 0 \text{ for } 1 \leq j \leq r-1, \\
\dim_{\bF_{q}}(\ol{A}[\ol{P}_{r}^{\infty}]/\ol{P}_{r}\ol{A}[\ol{P}_{r}^{\infty}]) = l_{H}
\end{array}
\right).
\end{align*}

\

Hence, applying Theorem \ref{main2} and Lemma \ref{CL2}, this leads to

\begin{align*}
\lim_{n \ra \infty}&\Prob_{A \in \Mat_{n}(R/\mf{m}^{N+1})}\left(
\begin{array}{c}
\coker(P_{j}(A)) = 0 \text{ for } 1 \leq j \leq r-1, \\
\coker(P_{r}(A)) \simeq H
\end{array}
\right) \\
&= \frac{q^{l_{H}^{2}}\prod_{i=1}^{l_{H}}( 1 - q^{-i} )^{2}}{ |\Aut_{R}(H)| } \cdot \frac{q^{-l_{H}^{2}}\prod_{i=1}^{\infty}(1 - q^{-i})}{\prod_{i=1}^{l_{H}}(1 - q^{-i})^{2}} \cdot \prod_{j=1}^{r-1}\prod_{i=1}^{\infty}(1-q^{-i\deg(P_{j})}) \\
&= \frac{1}{ |\Aut_{R}(H)| } \prod_{j=1}^{r}\prod_{i=1}^{\infty}(1-q^{-i\deg(P_{j})}),
\end{align*}

\

noting that $\deg(P_{r}) = 1$. This finishes the proof.
\end{proof}

\

\section{Boreico's formula}

\hspace{3mm} We now introduce a formula due to Boreico, appearing in his proof of Theorem \ref{main1} (or Theorem 3.8.18 in \cite{Bor}). We will use this formula to see that Theorem \ref{main1} and Theorem \ref{main2} give no more information than proving Theorem \ref{main1x} and Theorem \ref{main2x}. Any reader who only cares about proofs of our main theorems (Theorems \ref{main1x}, \ref{main2x}, and \ref{main3x} as well as Theorems \ref{main1} and \ref{main2}) can skip this section because merely proving them will not require Boreico's formula. However, the remark following Lemma \ref{Bor} explains how the formula can be used to prove a special case of Theorem \ref{main1}.

\

\begin{lem}[Boreico]\label{Bor} Fix any distinct monic irreducible polynomials $\ol{P}_{1}(t), \dots, \ol{P}_{r}(t) \in \bF_{q}[t]$. For $1 \leq j \leq r$, fix a finite $\ol{P}_{j}^{\infty}$-torsion module $H_{j}$ over $\bF_{q}[t]$ and let $h_{j} := \dim_{\bF_{q}}(H_{j})$. If $n \geq h_{1} + \cdots + h_{r}$, then we have

\begin{align*}
&\Prob_{\ol{A} \in \Mat_{n-(h_{1} + \cdots + h_{r})}(\bF_{q})}\left(\begin{array}{c}
\ol{A}[\ol{P}_{j}^{\infty}] = 0 \\
\text{ for } 1 \leq j \leq r
\end{array}\right) \\
&= \left(\frac{|\Aut_{\bF_{q}[x]}(H_{1})| \cdots |\Aut_{\bF_{q}[x]}(H_{r})|}{\prod_{i=n-(h_{1} + \cdots + h_{r}) + 1}^{n}(1 - q^{-i})}\right) \Prob_{\ol{A} \in \Mat_{n}(\bF_{q})}\left(\begin{array}{c}
\ol{A}[\ol{P}_{j}^{\infty}] \simeq H_{j} \\
\text{ for } 1 \leq j \leq r
\end{array}\right).
\end{align*}
\end{lem}

\

\begin{rmk} Lemma \ref{Bor} reduces Theorem \ref{main2} to the special case where $H_{1} = \cdots = H_{r} = 0$, and it can be similarly applied to reduce the task of proving Theorem \ref{main1}. In particular, if $r = 1$ and $\ol{P}_{1}(t) = t$, then writing $H = H_{1}$ and $h = h_{1} \leq n$, we can apply Lemma \ref{Bor} to compute

\begin{align*}
\Prob_{\ol{A} \in \Mat_{n}(\bF_{q})}(\ol{A}[t^{\infty}] \simeq H) &= \frac{1}{|\Aut_{\bF_{q}[t]}(H)|}\Prob_{\ol{A} \in \Mat_{n-h}(\bF_{q})}(\ol{A}[t^{\infty}] = 0)\prod_{i=n-h+1}^{n}(1 - q^{-i}) \\
& = \frac{1}{|\Aut_{\bF_{q}[t]}(H)|}\frac{|\GL_{n-h}(\bF_{q})|}{|\Mat_{n-h}(\bF_{q})|}\prod_{i=n-h+1}^{n}(1 - q^{-i}) \\
& = \frac{1}{|\Aut_{\bF_{q}[t]}(H)|}\prod_{i=1}^{n}(1 - q^{-i}),
\end{align*}

\

which proves a special case of Theorem \ref{main1}.
\end{rmk}

\

\hspace{3mm} We have seen that Theorem \ref{main1} and Theorem \ref{main2} imply Theorem \ref{main1x} and Theorem \ref{main2x}. The following corollary of the above formula lets us see that the converse can be easily achieved.

\

\begin{cor}\label{formula} Let $P_{1}(t), \dots, P_{r}(t) \in R[t]$ be monic polynomials such that the reduction modulo $\mf{m}$ gives distinct irreducible polynomials $\ol{P}_{1}(x), \dots, \ol{P}_{r}(x) \in \bF_{q}[x]$. For $1 \leq j \leq r$, fix a finite $\ol{P}_{j}^{\infty}$-torsion module $H_{j}$ over $\bF_{q}[x]$ and let $h_{j} := \dim_{\bF_{q}}(H_{j})$. If $n \geq h_{1} + \cdots + h_{r}$, then we have

\begin{align*}
&\Prob_{A \in \Mat_{n-(h_{1} + \cdots + h_{r})}(R)}\left(\begin{array}{c}
\coker(P_{j}(A)) = 0 \\
\text{ for } 1 \leq j \leq r
\end{array}\right) \\
&= \left(\frac{|\Aut_{\bF_{q}[x]}(H_{1})| \cdots |\Aut_{\bF_{q}[x]}(H_{r})|}{\prod_{i=n-(h_{1} + \cdots + h_{r}) + 1}^{n}(1 - q^{-i})}\right) \Prob_{\ol{A} \in \Mat_{n}(\bF_{q})}\left(\begin{array}{c}
\ol{A}[\ol{P_{j}}^{\infty}] \simeq H_{j} \\
\text{ for } 1 \leq j \leq r
\end{array}\right).
\end{align*}
\end{cor}

\begin{proof} This follows from applying Lemma \ref{Haar} with $N = 0$ to Lemma \ref{Bor}.
\end{proof}

\

\hspace{3mm} We now prove Lemma \ref{Bor}. This proof is due to Boreico (p.109 of \cite{Bor}).

\

\begin{proof}[Proof of Lemma \ref{Bor}]  In this proof, we write $A \in \Mat_{n}(\bF_{q})$ instead of $\ol{A}$ and $P_{j}$ replacing $\ol{P}_{j}$ for the sake of convenience. Let $H := H_{1} \op \cdots \op H_{r}$, and $h := \dim_{\bF_{q}}(H) = h_{1} + \cdots + h_{r}$. The key observation is that the number of $A \in \Mat_{n}(\bF_{q})$ such that $A[P_{j}^{\infty}] \simeq H_{j}$ for $1 \leq j \leq r$ is equal to the number of triples $(V, \phi, \psi)$ where

\begin{itemize}
	\item $V$ is an $\bF_{q}$-linear subspace of $\bF_{q}^{n}$ with dimension $h$;
	\item $\phi \in \End_{\bF_{q}}(V)$ such that $(\phi \acts V) \simeq H$ as $\bF_{q}[t]$-modules;
	\item $\psi \in \End_{\bF_{q}}(\bF_{q}^{n})$ such that $\psi|_{V} = \phi$ and

$$\bop_{i = 1}^{r}(\psi \acts \bF_{q}^{n})[P_{j}^{\infty}] \simeq (\phi \acts V)$$
\end{itemize}

\

as $\bF_{q}[t]$-modules, where the direct sum is internally taken in $\bF_{q}^{n}$. For any given $(V, \phi)$ with $(\phi \acts V) \simeq H$, the number of $\psi$ satisfying the above conditions is equal to the number of matrices of the form

$$\begin{bmatrix}
H & B \\
0 & C
\end{bmatrix},$$

\

where $H$ also means the $h \times h$ rational canonical form of the $\bF_{q}[t]$-module $H$, while $B$ is any $h \times (n - h)$ matrix and $C \in \Mat_{n - h}(\bF_{q})$ such that $P_{1}(C) \cdots P_{r}(C) \in \GL_{n-h}(\bF_{q})$. The number of such matrices is

$$q^{h(n-h)} |\{C \in \Mat_{n-h}(\bF_{q}) : C[P_{j}^{\infty}] = 0 \text{ for } 1 \leq j \leq r\}|.$$

\

It remains to count the number of $(V, \phi)$ described above. Given any $\bF_{q}$-linear injection $\alpha : H \hra \bF_{q}^{n}$, we may get such a pair by taking $V = \alpha(H)$ and $\phi = \alpha t \alpha^{-1}|_{V}$, where $t$ here means the $\bF_{q}$-linear endomorphism of $H$ given by the action of $t$. Every pair $(V, \phi)$ with $(\phi \acts V) \simeq H = (t \acts H)$ arises this way, and any two $\alpha, \beta \in \Inj_{\bF_{q}}(H, \bF_{q}^{n})$ give rise to the same pair precisely when

\begin{itemize}
	\item $\alpha(H) = \beta(H)$ (so that we call it $V$) and
	\item $\alpha t \alpha^{-1}|_{V} = \beta t \beta^{-1}|_{V}$.
\end{itemize}

The second condition can be restated as $\alpha^{-1}|_{V} \beta \in \Aut_{\bF_{q}[t]}(H)$. By taking $\eta = \alpha^{-1}|_{V} \beta$, we see that $\alpha, \beta \in \Inj_{\bF_{q}}(H, \bF_{q}^{n})$ give the same pair $(V, \phi)$ if and only if there is $\eta \in \Aut_{\bF_{q}[t]}(H)$ such that $\beta = \alpha\eta$. Thus, the set $\Inj_{\bF_{q}}(H, \bF_{q})/\Aut_{\bF_{q}[t]}(H)$ of orbits under the right action $\Inj_{\bF_{q}}(H, \bF_{q}) \righttoleftarrow \Aut_{\bF_{q}[t]}(H)$, given by the pre-composition, parametrizes the pairs $(V, \phi)$ such that $V$ is an $h$-dimensional subspace of $\bF_{q}^{n}$ and $(\phi \acts V) \simeq H$. This is a free action, so by Burnside's lemma, we have

$$|\Inj_{\bF_{q}}(H, \bF_{q})/\Aut_{\bF_{q}[t]}(H)| = \frac{|\Inj_{\bF_{q}}(H, \bF_{q})|}{|\Aut_{\bF_{q}[t]}(H)|} = \frac{1}{|\Aut_{\bF_{q}[t]}(H)|}(q^{n} - 1)(q^{n} - q) \cdots (q^{n}- q^{h-1})$$

\

because we have assumed that $n \geq h$. Combining altogether, we have

\begin{align*}
\left|\left\{
\begin{array}{c}
A \in \Mat_{n}(\bF_{q}) : \\
A[P_{j}^{\infty}] = H_{j} \text{ for } 1 \leq j \leq r
\end{array}
\right\}\right|
= \frac{q^{h(n-h)}(q^{n}-1)(q^{n}-q) \cdots (q^{n}-q^{h-1})}{|\Aut_{\bF_{q}[t]}(H)|}\left|\left\{
\begin{array}{c}
C \in \Mat_{n-h}(\bF_{q}) : \\
C[P_{j}^{\infty}] = 0 \text{ for } 1 \leq j \leq r
\end{array}
\right\}\right|.
\end{align*}

\

Dividing by $q^{n^{2}} = |\Mat_{n}(\bF_{q})|$, we get

\begin{align*}
\Prob_{A \in \Mat_{n}(\bF_{q})}
\left(
\begin{array}{c}
A[P_{j}^{\infty}] = H_{j} \\
\text{ for } 1 \leq j \leq r
\end{array}
\right)
&= \frac{q^{-(n-h)(n-h)}\prod_{i=n-h+1}^{n}(1-q^{-i})}{|\Aut_{\bF_{q}[t]}(H)|}
\left|\left\{
\begin{array}{c}
C \in \Mat_{n-h}(\bF_{q}) : \\
C[P_{j}^{\infty}] = 0 \text{ for } 1 \leq j \leq r
\end{array}
\right\}\right| \\
&= \frac{\prod_{i=n-h+1}^{n}(1-q^{-i})}{|\Aut_{\bF_{q}[t]}(H)|}
\Prob_{C \in \Mat_{n-h}(\bF_{q})}\left(\begin{array}{c}
C[P_{j}^{\infty}] = 0 \\
\text{ for } 1 \leq j \leq r
\end{array}
\right).
\end{align*}

\

Since $\Aut_{\bF_{q}[t]}(H) \simeq \Aut_{\bF_{q}[t]}(H_{1}) \times \cdots \times \Aut_{\bF_{q}[t]}(H_{r})$, this finishes the proof.
\end{proof}

\

\section{Useful lemmas for Theorem \ref{main1} and \ref{main2}} 

\hspace{3mm} Our main tool in proving Theorem \ref{main1} and Theorem \ref{main2} is a generating function that encodes information about similarity classes in $\Mat_{n}(\bF_{q})$.

\

\subsection{Cycle index}\label{ci} Every matrix $A \in \Mat_{n}(\bF_{q})$ gives rise to an $\bF_{q}[t]$-module structure on $\bF_{q}^{n}$, and up to an $\bF_{q}[t]$-isomorphism, it is

\[H_{P_{1}, \ld^{(1)}} \op \cdots \op H_{P_{r}, \ld^{(r)}}.\]

\

where $P_{i}(t) \in \bF_{q}[t]$ are monic irreducible polynomials and $\ld^{(i)} = (\ld_{1}^{(i)}, \dots, \ld_{l_{i}}^{(i)})$ are nonempty partitions with

$$H_{P_{i}, \ld^{(i)}} := \bF_{q}[t]/(P_{i}(t)^{\ld_{i,1}}) \op \cdots \op \bF_{q}[t]/(P_{i}(t)^{\ld_{i,l_{i}}})$$

\

as long as $n \geq 1$. For $n = 0$, we have $r = 0$, and this is consistent with the fact that we only have the zero module for this case. Up to a permutation, these $H_{P_{i}, \ld^{(i)}}$ characterize the similarity class of $A$. For any monic irreducible $P = P(t) \in \bF_{q}[t]$, we denote by $\mu_{P}(A)$ the partition associated to the $P$-part of $A$ or to the isomorphism class of the $\bF_{q}[t]$-module $A \acts \bF_{q}^{n}$. More specifically, in the above notation, we have

\[\mu_{P_{i}}(A) = \ld^{(i)} =  (\ld_{1}^{(i)}, \dots, \ld_{l_{i}}^{(i)})\]

\

and $\mu_{P}(A) = \es$ when $P \neq P_{i}$ for all $i$. Write $|\bA^{1}_{\bF_{q}}| = |\Spec(\bF_{q}[t])|$ to mean the set of all monic irreducible polynomials in $\bF_{q}[t]$. As the notation suggests, $|\bA^{1}_{\bF_{q}}|$ can be seen as the set of closed points of the affine line $\bA^{1}_{\bF_{q}} = \Spec(\bF_{q}[t])$ over $\bF_{q}$. For each nonempty partition $\nu$ and $P \in |\bA^{1}_{\bF_{q}}|$, we consider a formal variable $x_{P, \nu}$. For the empty partition $\es$, we put $x_{P, \es} := 1$. As in Section \ref{main}, we write $\mc{P}$ to mean the set of all partitions of non-negative integers, where the only partition for $0$ is $\es$.

\

\hspace{3mm} From the structure theorem about finitely generated modules over $\bF_{q}[t]$, which is a PID, and the Chinese remainder theorem, we note that for any two matrices $A, B \in \Mat_{n}(\bF_{q})$, the following are equivalent:

\be
	\item $A$ and $B$ are similar;
	\item $A$ and $B$ give the isomorphic $\bF_{q}[t]$-module structures on $\bF_{q}^{n}$;
	\item $A$ and $B$ are in the same orbit under the conjugate action $\GL_{n}(\bF_{q}) \acts \Mat_{n}(\bF_{q})$;
	\item $\mu_{P}(A) = \mu_{P}(B)$ for all $P \in |\bA^{1}_{\bF_{q}}|$;
	\item $\prod_{P \in |\bA^{1}_{\bF_{q}}|}x_{P,\mu_{P}(A)} = \prod_{P \in |\bA^{1}_{\bF_{q}}|}x_{P,\mu_{P}(B)}$.
\ee

\

\hspace{3mm} We define the \textbf{$n$-th cycle index of the conjugate action $\GL_{n}(\bF_{q}) \acts \Mat_{n}(\bF_{q})$} to be the polynomial

$$\mc{Z}([\Mat_{n}/\GL_{n}](\bF_{q}), \bs{x}) := \frac{1}{|\GL_{n}(\bF_{q})|}\sum_{A \in \Mat_{n}(\bF_{q})}\prod_{P \in |\bA^{1}_{\bF_{q}}|}x_{P, \mu_{P}(A)} \in \bQ[\bs{x}],$$

\

where $\bs{x} := (x_{P,\nu})$ is the sequence of formal variables $x_{P,\nu}$. We define the \textbf{$n$-th cycle index of the group $\GL_{n}(\bF_{q})$} by the analogous definition for the restricted conjugation action $\GL_{n}(\bF_{q}) \acts \GL_{n}(\bF_{q})$:

$$\mc{Z}(\GL_{n}(\bF_{q}), \bs{x}) := \frac{1}{|\GL_{n}(\bF_{q})|}\sum_{A \in \GL_{n}(\bF_{q})}\prod_{P \in |\bA^{1}_{\bF_{q}}|}x_{P, \mu_{P}(A)} \in \bQ[\bs{x}].$$

\

Notice that the irreducible polynomial $P(t) = t$ will not occur in the product above because for any $A \in \Mat_{n}(\bF_{q})$, saying that $A \in \GL_{n}(\bF_{q})$ is equivalent to saying $\mu_{t}(A) = \es$ (i.e., $A$ has no $t$-part).

\

\subsection{Useful lemmas} We will introduce three lemmas useful for proving Theorems \ref{main1} and \ref{main2}. The first one is due to Stong, who introduced the cycle index of the conjugate action $\GL_{n}(\bF_{q}) \acts \Mat_{n}(\bF_{q})$.

\

\begin{lem}[Lemma 1 in \cite{Sto}]\label{Sto} We have

\begin{align*}
\sum_{n=0}^{\infty}\mc{Z}([\Mat_{n}/\GL_{n}](\bF_{q}), \bs{x})u^{n} &= \sum_{n=0}^{\infty}\sum_{A \in \Mat_{n}(\bF_{q})}\left(\frac{\prod_{P \in |\bA^{1}_{\bF_{q}}|} x_{P, \mu_{P}(A)}}{|\GL_{n}(\bF_{q})|}\right) u^{n} \\ 
&=  \prod_{P \in |\bA^{1}_{\bF_{q}}|}\sum_{\nu \in \mc{P}}\frac{x_{P,\nu}u^{|\nu|\deg(P)}}{|\Aut_{\bF_{q}[t]}(H_{P,\nu})|}
\end{align*}

\

in $\bQ[\bs{x}]\llb u \rrb$.
\end{lem}

\

\hspace{3mm} The result above proves the following lemma due to Kung, who introduced the cycle index of $\GL_{n}(\bF_{q})$:

\begin{lem}[Lemma 1 in \cite{Kun}]\label{Kun} We have

\begin{align*}\sum_{n=0}^{\infty}\mc{Z}(\GL_{n}(\bF_{q}), \bs{x})u^{n} &= \sum_{n=0}^{\infty}\sum_{A \in \GL_{n}(\bF_{q})} \left(\frac{\prod_{P \in |\bA^{1}_{\bF_{q}}|} x_{P, \mu_{P}(A)}}{|\GL_{n}(\bF_{q})|}\right) u^{n} \\
&= \prod_{\substack{P \in |\bA^{1}_{\bF_{q}}|, \\ P(t) \neq t}}\sum_{\nu \in \mc{P}}\frac{x_{P,\nu}u^{|\nu|\deg(P)}}{|\Aut_{\bF_{q}[t]}(H_{P,\nu})|}
\end{align*}

\

in $\bQ[\bs{x}]\llb u \rrb$.
\end{lem}

\begin{proof} If we take $x_{t, \nu} = 0$ for all nonempty partitions $\nu$ in the expression

\[\mc{Z}([\Mat_{n}/\GL_{n}](\bF_{q}), \bs{x}) = \frac{1}{|\GL_{n}(\bF_{q})|}\sum_{A \in \Mat_{n}(\bF_{q})}\prod_{P \in |\bA^{1}_{\bF_{q}}|}x_{P, \mu_{P}(A)},\]

\

we get $\mc{Z}(\GL_{n}(\bF_{q}), \bs{x})$ because any square matrix is invertible if and only if it does not have $0$ eigenvalue (or equivalently, if it does not have any invariant factor divisible by $t$). Thus, Lemma \ref{Sto} implies the result.
\end{proof}

\

\hspace{3mm} The following is the third lemma we need, due to Stong (from our best knowledge). This lemma serves a crucial role in the proofs of Theorem \ref{main1} and Theorem \ref{main2}, and Stong's proof relies on the fact that there are $q^{n(n-1)}$ nilpotent matrices in $\Mat_{n}(\bF_{q})$, a famous result of Fine and Herstein \cite{FH}.

\

\begin{lem}[Proposition 19 in \cite{Sto}]\label{key} For any $P \in |\bA^{1}_{\bF_{q}}|$, we have

\[\sum_{\nu \in \mc{P}} \frac{y^{|\nu|}}{|\Aut_{\bF_{q}[t]}(H_{P,\nu})|} = \prod_{i=1}^{\infty}\frac{1}{1 - q^{-i\deg(P)}y} \in \bQ\llb y \rrb.\]
\end{lem}

\

\begin{rmk} \label{sum=1} Using Macdonald's result in (1.6) on p.181 in \cite{Mac}, Lemma \ref{key} implies that for any DVR $(R, \mf{m})$ with $R/\mf{m} = \bF_{q}$, we have

\[\sum_{H \in \Mod_{R}^{<\infty}} \frac{y^{|\nu|}}{|\Aut_{R}(H)|} = \prod_{i=1}^{\infty}\frac{1}{1 - q^{-i}y} \in \bQ\llb y \rrb.\]

\

Hence, taking $y = 1$, this proves that the assignment $\{H\} \mapsto |\Aut_{R}(H)|^{-1}\prod_{i=1}^{\infty}(1 - q^{-i})$ is indeed a probability measure on $\Mod_{R}^{<\infty}$.
\end{rmk}

\

\section{Proofs of Theorem \ref{main1} and \ref{main2}} 

\hspace{3mm} In this section, we provide proofs of Theorem \ref{main1} and Theorem \ref{main2}. Due to Section \ref{reduction}, this will finish the proofs of Theorem \ref{main1x}, Theorem \ref{main2x}, and Theorem \ref{main3x}.

\

\subsection{Proof of Theorem \ref{main1}}\label{limit} We first deal with the sequence $(b_{n}(d))_{n \in \bZ_{\geq 0}}$ appearing in Theorem \ref{main1x} and Theorem \ref{main1} as well as some convergences of relevant infinite products of formal power series. Such a product needs to be treated with care because its expansion leads to a power series whose coefficients are given by infinite sums.

\

\hspace{3mm} First, fix $0 \leq t < 1$. The sequence

$$\prod_{i=1}^{n}(1 - t^{i}) = (1 - t)(1 - t^{2}) \cdots (1 - t^{n})$$

\

is decreasing in $n$, while it is bounded below by $0$. Thus, the sequence converges in $\bR$. Since $0 \leq t < 1$, an application of Theorem 15.4 of \cite{Rud} ensures that the limit of this product as $n \ra \infty$ is nonzero. In particular, taking $t = q^{-1}$, we see $\prod_{i=1}^{\infty}(1 - q^{-i}) > 0$ makes sense, and so does

$$\prod_{i=1}^{\infty}\frac{1 - q^{-di}}{1 - q^{-i}} := \frac{\prod_{i=1}^{\infty}(1 - q^{-di})}{\prod_{i=1}^{\infty}(1 - q^{-i})}$$

\

for any $d \in \bZ_{\geq 1}$. The power series $\sum_{i=1}^{\infty}q^{-i}u$ has radius of convergence $q$ at $u = 0$. Hence, by taking $f_{i}(u) = 1 - q^{-i}u$ in Theorem 15.6 of \cite{Rud}, we see that the product

$$\prod_{i=1}^{\infty}f_{i}(u) = \prod_{i=1}^{\infty}(1 - q^{-i}u)$$

\

converges uniformly on any compact subsets of $\{u \in \bC : |u| < q\}$. The power series $\sum_{i=1}^{\infty}q^{-di}u^{d}$ also has radius of convergence $q$ at $u = 0$, so we may apply the same theorem to deduce that the product

$$\prod_{i=1}^{\infty}(1 - (q^{-i}u)^{d})$$
 
\

converges uniformly on any compact subsets of $\{u \in \bC : |u| < q\}$. This implies that both products are holomorphic in $\{u \in \bC : |u| < q\}$, and hence so is their ratio (as none of them vanishes in the specified open disc of $\bC$ with radius $q$). Thus, we may rewrite it as a power series

$$\prod_{i=1}^{\infty}\frac{1 - (q^{-i}u)^{d}}{1 - q^{-i}u} = a_{0}(d) + a_{1}(d)u + a_{2}(d)u^{2} + \cdots,$$

\

whose radius of convergence is $q$ at $u = 0$. Thus, we can evaluate both sides at $u = 1 < q$ to have:

$$\prod_{i=1}^{\infty}\frac{1 - q^{-id}}{1 - q^{-i}} = a_{0}(d) + a_{1}(d) + a_{2}(d) + \cdots.$$

\

Since the only holomorphic function in the open disc with a limit point in its zero set (in the open disc) must be the zero function, we must have the same identity

$$\prod_{i=1}^{\infty}\frac{1 - (q^{-i}u)^{d}}{1 - q^{-i}u} = a_{0}(d) + a_{1}(d)u + a_{2}(d)u^{2} + 
\cdots$$

\

in $\bC\llb u \rrb$ as well, where we take $u$ to be formal. Therefore, in $\bC\llb u \rrb$, we have

\begin{align*}
b_{0}(d) + b_{1}(d)u + b_{2}(d)u^{2} + \cdots &= \frac{1}{1 - u}\prod_{i=1}^{\infty}\frac{1 - (q^{-i}u)^{d}}{1 - q^{-i}u} \\
&= (1 + u + u^{2} + \cdots)(a_{0}(d) + a_{1}(d)u + a_{2}(d)u^{2} + \cdots) \\
&= a_{0}(d) + (a_{0}(d) + a_{1}(d))u + (a_{0}(d) + a_{1}(d) + a_{2}(d))u^{2} + \cdots.
\end{align*}

\

This implies that $b_{n}(d) = a_{0}(d) + a_{1}(d) + \cdots + a_{n}(d)$, so

$$\lim_{n \ra \infty}b_{n}(d) = a_{0}(d) + a_{1}(d) + \cdots =  \prod_{i=1}^{\infty}\frac{1 - q^{-id}}{1 - q^{-i}},$$

\

and this proves the last parts of Theorem \ref{main1x} and Theorem \ref{main1}. Thus, we only need to show the statement of Theorem \ref{main1} before we take the limit $n \ra \infty$ to finish its proof.

\

\begin{proof}[Proof of Theorem \ref{main1}] We denote by $P_{0}$ to mean $P$ in the statement for this proof. We may assume that

$$H = H_{P_{0},\ld} = \bF_{q}[t]/(P_{0}(t))^{\ld_{1}} \op \cdots \op \bF_{q}[t]/(P_{0}(t))^{\ld_{l}}$$

\

for some fixed partition $\ld = (\ld_{1}, \dots, \ld_{l}) \in \mc{P}$. The case $\ld = \es$ (i.e., $H = 0$) turns out to be the most important. For this, it is enough to show that

\[b_{n}(\deg(P_{0})) = \frac{|\{A \in \Mat_{n}(\bF_{q}) : \mu_{P_{0}}(A) = \es\}|}{|\GL_{n}(\bF_{q})|}.\]

\

To see this, let $a_{n}(P_{0})$ be the expression on the right-hand side. Take $x_{P_{0}, \nu} = 0$ for all nonempty $\nu$ and $x_{P,\nu} = 1$ for all $P \neq P_{0}$ in Lemma \ref{Sto}, which leads to

\begin{align*}
\sum_{n=0}^{\infty}a_{n}(P_{0}) u^{n} &= \prod_{\substack{P \in |\bA^{1}_{\bF_{q}}|, \\ P(t) \neq P_{0}(t)}}\sum_{\nu \in \mc{P}} \frac{u^{|\nu|\deg(P)}}{|\Aut_{\bF_{q}[t]}(H_{P,\nu})|} \\
&= \left(\sum_{\nu \in \mc{P}} \frac{u^{|\nu|\deg(P_{0})}}{|\Aut_{\bF_{q}[t]}(H_{P_{0},\nu})|}\right)^{-1}\left(\sum_{\nu \in \mc{P}}\frac{u^{|\nu|}}{|\Aut_{\bF_{q}[t]}(H_{t,\nu})|}\right)\prod_{\substack{P \in |\bA^{1}_{\bF_{q}}|, \\ P(t) \neq t}}\sum_{\nu \in \mc{P}} \frac{u^{|\nu|\deg(P)}}{|\Aut_{\bF_{q}[t]}(H_{P,\nu})|} \\
&= \left( \prod_{i=1}^{\infty}  \frac{1 - (q^{-i}u)^{\deg(P_{0})}}{1 - q^{-i}u} \right)\left(\frac{1}{1 - u}\right) \\
&= \prod_{i=1}^{\infty}  \frac{1 - (q^{-i}u)^{\deg(P_{0})}}{1 - q^{1-i}u},
\end{align*}

\

where we applied Lemma \ref{Kun} and Lemma \ref{key} as well. This shows that $a_{n}(P_{0}) = b_{n}(\deg(P_{0}))$ by definition of $b_{n}(d)$ in the statement of Theorem \ref{main1x} and Theorem \ref{main1}.

\

\hspace{3mm} Now, we may assume that the partition $\ld = (\ld_{1}, \dots, \ld_{l})$ is nonempty (i.e., $l > 0$). In Lemma \ref{Sto}, take $x_{P, \nu} = 1$ on both sides for $P \neq P_{0}$ to get

\[\sum_{n=0}^{\infty}\sum_{A \in \Mat_{n}(\bF_{q})} \frac{x_{P_{0}, \mu_{P_{0}}(A)}}{|\GL_{n}(\bF_{q})|} u^{n} =  \left(\sum_{\nu \in \mc{P}}\frac{x_{P_{0},\nu}u^{|\nu|\deg(P_{0})}}{|\Aut_{\bF_{q}[t]}(H_{P_{0},\nu})|}\right)\left(\prod_{P \neq P_{0}}\sum_{\nu \in \mc{P}}\frac{u^{|\nu|\deg(P)}}{|\Aut_{\bF_{q}[t]}(H_{P,\nu})|}\right).\]

\

Next, we take $x_{P_{0}, \nu} = 0$ for all nonempty $\nu \neq \ld$ and $x_{P_{0}, \ld} = 1$. Then

\begin{align*}
1 &+ \sum_{n=1}^{\infty}\left(\frac{|\{A \in \Mat_{n}(\bF_{q}) : \mu_{P_{0}}(A) = \ld \text{ or } \es\}|}{|\GL_{n}(\bF_{q})|}\right) u^{n} \\
&=  \left(1 + \frac{u^{|\ld|\deg(P_{0})}}{|\Aut_{\bF_{q}[t]}(H_{P_{0},\ld})|}\right) \left(\prod_{P \neq P_{0}}\sum_{\nu \in \mc{P}}\frac{u^{|\nu|\deg(P)}}{|\Aut_{\bF_{q}[t]}(H_{P,\nu})|}\right) \\
&=  \left(1 + \frac{u^{|\ld|\deg(P_{0})}}{|\Aut_{\bF_{q}[t]}(H_{P_{0},\ld})|}\right)\left(\prod_{P \in |\bA^{1}_{\bF_{q}}|}\sum_{\nu \in \mc{P}}\frac{u^{|\nu|\deg(P)}}{|\Aut_{\bF_{q}[t]}(H_{P,\nu})|}\right)\left(\sum_{\nu \in \mc{P}}\frac{u^{|\nu|\deg(P_{0})}}{|\Aut_{\bF_{q}[t]}(H_{P_{0},\nu})|}\right)^{-1} \\
&=  \left(1 + \frac{u^{|\ld|\deg(P_{0})}}{ |\Aut_{\bF_{q}[t]}(H_{P_{0},\ld})| }\right)\left(\prod_{P(t) \neq t}\sum_{\nu \in \mc{P}}\frac{u^{|\nu|\deg(P)}}{ |\Aut_{\bF_{q}[t]}(H_{P,\nu})| }\right) \left(\sum_{\nu \in \mc{P}}\frac{u^{|\nu|}}{ |\Aut_{\bF_{q}[t]}(H_{(t),\nu})| }\right) \left(\sum_{\nu \in \mc{P}}\frac{u^{|\nu|\deg(P_{0})}}{ |\Aut_{\bF_{q}[t]}(H_{P_{0},\nu})| }\right)^{-1} \\
&=  \left(1 + \frac{u^{|\ld|\deg(P_{0})}}{ |\Aut_{\bF_{q}[t]}(H_{P_{0},\ld})| }\right)\left(\frac{1}{1-u}\right) \left(\prod_{i=1}^{\infty} \frac{1-(q^{-i}u)^{\deg(P_{0})}}{1-q^{-i}u}\right),
\end{align*}

\

applying Lemma \ref{Kun} and Lemma \ref{key}. Thus, we have

\begin{align*}
1 + \sum_{n=1}^{\infty}&\left(\frac{|\{A \in \Mat_{n}(\bF_{q}) : \mu_{P_{0}}(A) = \ld \text{ or } \es\}|}{|\GL_{n}(\bF_{q})|}\right) u^{n} \\
&= (1 + cu^{h})(1 + b_{1}u +b_{2}u^{2} + b_{3}u^{3} + \cdots) \\
&= 1 + b_{1}u + b_{2}u + \cdots + b_{h-1}u^{h-1} + (b_{h} + c)u^{h} + (b_{h+1} + cb_{1})u^{h+1} + (b_{h+2} + cb_{2})u^{h+2} + \cdots, 
\end{align*}

where

\begin{itemize}
	\item $c = |\Aut_{\bF_{q}[t]}(H_{P_{0},\ld})|^{-1} = |\Aut_{\bF_{q}[t]}(H)|^{-1}$,
	\item $b_{n} = b_{n}(\deg(P_{0}))$, and
	\item $h = |\ld|\deg(P_{0})$ = $\dim_{\bF_{q}}(H)$.
\end{itemize}

\

Thus, continuing the previous computations, since we have established that

\[b_{n} = \frac{|\{A \in \Mat_{n}(\bF_{q}) : \mu_{P_{0}}(A) = \es\}|}{|\GL_{n}(\bF_{q})|},\]

\

we have (as $b_{0} = 1$)

\[\frac{|\{A \in \Mat_{n}(\bF_{q}) : \mu_{P_{0}}(A) = \ld\}|}{|\GL_{n}(\bF_{q})|} = \left\{
	\begin{array}{ll}
	cb_{n-h} = |\Aut_{\bF_{q}[t]}(H)|^{-1}b_{n-h}(\deg(P_{0})) & \mbox{if } n \geq h = |\ld|\deg(P_{0}), \\
	0 & \mbox{if } n < h = |\ld|\deg(P_{0}).
	\end{array}\right.\]
	
\

By multiplying

\begin{align*}
\frac{|\GL_{n}(\bF_{q})|}{|\Mat_{n}(\bF_{q})|} &= \frac{(q^{n} - 1)(q^{n} - q) \cdots (q^{n} - q^{n-1})}{q^{n^{2}}} \\
&= (1 - q^{-1})(1 - q^{-2}) \cdots (1 - q^{-n})
\end{align*}

\

both sides, we finish the proof.
\end{proof}

\

\subsection{Proof of Theorem \ref{main2}}\label{pf} Before the proof, we define one more terminology that will enable us to write a clearer proof. Fix any subset $X \sub \bA^{1}_{\bF_{q}} = \Spec(\bF_{q}[t])$. We define the \textbf{cycle index of $X$} (relative to $\bA^{1}_{\bF_{q}}$) as follows:

$$\hat{\bs{Z}}(X, \bs{x}, u) = \prod_{P \in X \cap |\bA^{1}_{\bF_{q}}|}\sum_{\nu \in \mc{P}}\frac{x_{P,\nu}u^{|\nu|\deg(P)}}{|\Aut_{\bF_{q}[t]}(H_{P,\nu})|},$$

\

where each $P \in |\bA^{1}_{\bF_{q}}|$ simultaneously means a monic irreducible polynomial or the maximal ideal $(P(t))$ of $\bF_{q}[t]$ generated by it (i.e., a closed point of $\bA^{1}_{\bF_{q}}$). Note that by Lemma \ref{Sto}, we have

$$\hat{\bs{Z}}(\bA^{1}_{\bF_{q}}, \bs{x}, u) = \sum_{n=0}^{\infty}\mc{Z}([\Mat_{n}/\GL_{n}](\bF_{q}), \bs{x})u^{n}.$$

\

That is, the cycle index of the affine line $\bA^{1}_{\bF_{q}}$ is the generating function for the $n$-th cycle index of the conjugate action $\GL_{n}(\bF_{q}) \acts \Mat_{n}(\bF_{q})$ for all $n \in \bZ_{\geq 0}$. Another important example is

$$\hat{\bs{Z}}(\{P\}, \bs{x}, u) = \sum_{\nu \in \mc{P}}\frac{x_{P,\nu}u^{|\nu|\deg(P)}}{|\Aut_{\bF_{q}[t]}(H_{P,\nu})|},$$

\

where $P \in |\bA^{1}_{\bF_{q}}|$. By definition, whenever we have finitely many $P_{1}, \dots, P_{r} \in X \cap |\bA^{1}_{\bF_{q}}|$, we have

$$\hat{\bs{Z}}(X, \bs{x}, u) = \hat{\bs{Z}}(X \sm \{P_{1}, \dots, P_{r}\}, \bs{x}, u) \hat{\bs{Z}}(\{P_{1}\}, \bs{x}, u) \cdots \hat{\bs{Z}}(\{P_{r}\}, \bs{x}, u).$$

\

Denote by $\hat{\bs{Z}}(X, u)$ what we get by taking all $x_{P,\nu} = 1$ in $\hat{\bs{Z}}(X, \bs{x}, u)$. Lemma \ref{Kun} implies that

$$\hat{\bs{Z}}(\bA^{1}_{\bF_{q}} \sm \{(t)\}, u) = 1 + u + u^{2} + \cdots = \frac{1}{1 - u}.$$

\

Finally, Lemma \ref{key} implies that for any $P \in |\bA^{1}_{\bF_{q}}|,$ we have

$$\hat{\bs{Z}}(\{P\}, u) = \sum_{\nu \in \mc{P}}\frac{u^{|\nu|\deg(P)}}{|\Aut_{\bF_{q}[t]}(H_{P,\nu})|} = \prod_{i=1}^{\infty}\frac{1}{1 - (q^{-i}u)^{\deg(P)}}.$$

\

We are now ready to give the proof of Theorem \ref{main2}.

\

\begin{proof}[Proof of Theorem \ref{main2}] We will use the notations and the arguments given above. Taking $x_{P,\nu} = 1$ for all $P \notin \{P_{1}, \dots, P_{r}\}$, while still denoting $\bs{x}$ to mean the sequence of variables after such evaluations, we have

\begin{align*}
\hat{\bs{Z}}(\bA^{1}, \bs{x}, u) &= \hat{\bs{Z}}(\bA^{1}_{\bF_{q}} \sm \{P_{1}, \dots, P_{r}\}, u) \hat{\bs{Z}}(\{P_{1}\}, \bs{x}, u) \cdots \hat{\bs{Z}}(\{P_{r}\}, \bs{x}, u) \\
&= \frac{\hat{\bs{Z}}(\bA^{1}_{\bF_{q}} \sm \{(t)\}, u) \hat{\bs{Z}}(\{(t)\}, u) \hat{\bs{Z}}(\{P_{1}\}, \bs{x}, u) \cdots \hat{\bs{Z}}(\{P_{r}\}, \bs{x}, u)}{ \hat{\bs{Z}}(\{P_{1}\}, u) \cdots \hat{\bs{Z}}(\{P_{r}\}, u)} \\
&= \left(\frac{1}{1 - u}\right)\frac{\hat{\bs{Z}}(\{(t)\}, u) \hat{\bs{Z}}(\{P_{1}\}, \bs{x}, u) \cdots \hat{\bs{Z}}(\{P_{r}\}, \bs{x}, u)}{ \hat{\bs{Z}}(\{P_{1}\}, u) \cdots \hat{\bs{Z}}(\{P_{r}\}, u)}.
\end{align*}

\

Without loss of generality, suppose that $\ld^{(1)}, \dots, \ld^{(m)}$ are nonempty, while $\ld^{(m+1)}, \dots, \ld^{(r)} = \es$, for some $0 \leq m \leq r$. In the above identity, take $x_{P_{j},\nu} = 0$ for nonempty $\nu$ not equal to $\ld^{(j)}$ while $x_{P_{j},\ld^{(j)}} = 1$ for $1 \leq j \leq r$. We will still write $\bs{x}$ to mean the sequence of variables after evaluations, although this is now just a sequence in $\{0, 1\}$. Arguing as in Section \ref{limit}, we may compute the limit of the coefficient of $u^{n}$ of the left-hand side as $n \ra \infty$, which results in

\begin{align*}
&\lim_{n \ra \infty} \Prob_{A \in \Mat_{n}(\bF_{q})} \left(\begin{array}{c}
\mu_{P_{j}}(A) \in \{\es, \ld^{(j)}\} \text{ for } 1 \leq j \leq m, \\
\mu_{P_{m+1}}(A) = \cdots = \mu_{P_{r}}(A) = \es
\end{array}\right) \\
&= \frac{ \hat{\bs{Z}}(\{P_{1}\}, \bs{x}, 1) \cdots \hat{\bs{Z}}(\{P_{r}\}, \bs{x}, 1)}{ \hat{\bs{Z}}(\{P_{1}\}, 1) \cdots \hat{\bs{Z}}(\{P_{r}\}, 1)} \\
&= \left[\prod_{j=1}^{m}\left(1 + \frac{1}{\#\Aut_{\bF_{q}}(H_{P,\ld^{(j)}})}\right) \prod_{i=1}^{\infty}(1 - q^{-i\deg(P_{j})})\right] \cdot \left[\prod_{j=m+1}^{r}\prod_{i=1}^{\infty}(1 - q^{-i\deg(P_{j})})\right],
\end{align*}

\

where the first identity used 

$$\hat{\bs{Z}}(\{(t)\}, 1) = \lim_{n \ra \infty}\frac{|\Mat_{n}(\bF_{q})|}{|\GL_{n}(\bF_{q})|}.$$

\

This finishes the proof, because one may either argue by induction on $m$ or see that the product measure on

$$\Mod_{\bF_{q}[t]_{(P_{1})}}^{<\infty} \times \cdots \times \Mod_{\bF_{q}[t]_{(P_{r})}}^{<\infty}$$

\

given by Cohen-Lenstra measures match the limiting probability with specified parts at $P_{1}, \dots, P_{r}$ for enough events that generate the finest $\sigma$-algebra on the product.
\end{proof}

\

\section{Another possible connection between Haar measure and $\bF_{q}$-matrices}

\hspace{3mm} The easiest case $N = 0$ for Lemma \ref{Haar} provides a connection between random matrices $A \in \Mat_{n}(\bF_{q} \llb t \rrb)$ with respect to the Haar measure and $\ol{A} \in \Mat_{n}(\bF_{q})$ with respect to the uniform distribution. There seem to be more mysterious connections between these two different random matrices. We illustrate one incidence here. Note that the $\bF_{q}[t]$-module structure on $\bF_{q}^{n}$ given by $\ol{A} \in \Mat_{n}(\bF_{q})$ is precisely $\coker(\ol{A} - tI_{n})$, where $\ol{A} - tI_{n}$ is viewed as a matrix over $\bF_{q}[t]$. Hence, if $\ol{A} - tI_{n}$ is viewed as a matrix over $\bF_{q} \llb t \rrb$, we have $\coker(\ol{A} - tI_{n}) \simeq (\ol{A} \acts \bF_{q}^{n})[t^{\infty}]$. Having said that, Theorem \ref{main1} with $P(t) = t$ can be restated as follows. Given any finite length $t^{\infty}$-torsion module $H$ over $\bF_{q}[t]$, we have

$$\Prob_{\ol{A} \in \Mat_{n}(\bF_{q})}(\coker(\ol{A} - tI_{n}) \simeq H) = \left\{
	\begin{array}{ll}
	\frac{1}{|\Aut_{\bF_{q}[t]}(H)|}\prod_{i=1}^{n}(1 - q^{-i}) & \mbox{if } n \geq \dim_{\bF_{q}}(H) \text{ and } \\
	0 & \mbox{if } n < \dim_{\bF_{q}}(H),
	\end{array}\right.$$

\

where the cokernels are taken over $\bF_{q} \llb t \rrb$. On the other hand, Proposition \ref{FW} with $R = \bF_{q} \llb t \rrb$ says

\begin{align*}
\Prob&_{A \in \Mat_{n}(\bF_{q} \llb t \rrb)} (\coker(A) \simeq H) \\
&= \left\{
	\begin{array}{ll}
	\frac{1}{|\Aut_{\bF_{q}[t]}(H)|} \left[ \prod_{i=1}^{n}\left( 1 - q^{-i} \right) \right]\left[ \prod_{j=n-l_{H}+1}^{n}\left( 1 - q^{-j} \right) \right] & \mbox{if } n \geq l_{H} = \dim_{\bF_{q}}(H/tH),  \\
	0 & \mbox{if } n < l_{H}.
	\end{array}\right.
\end{align*}

\

We may simultaneously consider both probabilities by using the projection map $\Mat_{n}(\bF_{q}\llb t \rrb) \tra \Mat_{n}(\bF_{q})$ given by $t \mapsto 0$. For each $\ol{A} \in \Mat_{n}(\bF_{q})$, the first probability considers the cokernel of a special representative $\ol{A} - tI_{n}$ in the fiber $\ol{A} + t \Mat_{n}(\bF_{q}\llb t \rrb)$ of $\ol{A}$. On the other hand, the second probability considers all the matrices in the fiber. The two probabilities are the same if and only if $H = 0$. It is interesting to note that regardless of the choice of $H$, both probabilities converge to the same Cohen-Lenstra distribution of $\bF_{q}\llb t \rrb$ as $n \ra \infty$, which is another incidence of ``universality'' we mentioned in the introduction. Giving more careful analysis on this discrepancy of the two different probabilities might be an interesting work in the near future.

%%%%%%%%%%%%%%%%%%%%%%%%%%%%%%%%%%%%%%%

%%%%%%%%%%%%%%%%%%%%%%%%%%%%%%%%%%%%%%%

\end{document}